\documentclass[11pt,reqno]{amsart}
\usepackage{amssymb}
\usepackage{amsthm}
\usepackage{eucal}
\usepackage{color,graphicx}

\def\e{{\varepsilon}}

\newcommand\R{\mathbb{R}}

\newcommand\N{\mathbb{N}}
\newcommand\T{\mathbb{T}}
\newcommand\Z{\mathbb{Z}}
\newcommand\NLI{{\rm I}}
\newcommand{\NLII}{{\rm I\!I}}
\newcommand{\NLIII}{{\rm I\!I\!I}}
\renewcommand\H{\mathcal{H}}

\newcommand\bna{\begin{eqnarray*}}
\newcommand\ena{\end{eqnarray*}}

\newcommand\bnan{\begin{eqnarray}}
\newcommand\enan{\end{eqnarray}}

\newcommand\bnp{\begin{proof}}
\newcommand\enp{\end{proof}}

\newcommand\bneq{\begin{eqnarray*}\left\lbrace \begin{array}{rcl}}
\newcommand\eneq{\end{array} \right.\end{eqnarray*}}
\newcommand\bneqn{\begin{eqnarray}\left\lbrace \begin{array}{rcl}}
\newcommand\eneqn{\end{array} \right.\end{eqnarray}}

\newcommand{\U}{\mathcal{U}}
\newcommand{\W}{\mathcal{W}}
\newcommand{\IW}{\mathcal{IW}}

\newcommand\G{\mathcal{G}}

\newcommand\nor[2]{\left\|#1\right\|_{#2}}

\newcommand\refp[1]{(\ref{#1})}

\newcommand\petito[1]{o(#1)}
\newcommand\grando[1]{\mathcal{O}(#1)}

\newtheorem{theorem}{Theorem}[section]
\newtheorem{remark}{Remark}[section]
\newtheorem{lemme}{Lemma}[section]

\newtheorem{prop}{Proposition}[section]

\numberwithin{equation}{section}
\newtheorem*{TA}{Theorem A}
\newtheorem*{TB}{Theorem B}
\newtheorem*{TC}{Theorem C}
\newtheorem*{TD}{Theorem D}
\newtheorem*{TE}{Theorem E}

\title{Control and Stabilization of the Benjamin-Ono equation  in $L^2(\mathbb T)$}

\author{Camille Laurent}
\address{Laboratoire Jacques-Louis Lions, 
Universit\'e Pierre et Marie Curie,
Bo{\'i}te courrier 187, 
75252 Paris Cedex 05, France}
\email{laurent@ann.jussieu.fr}

\author{Felipe Linares}
\address{Instituto de Matematica Pura e Aplicada,
Estrada Dona Castorina 110, Rio de Janeiro 22460-320, Brazil}
\email{linares@impa.br}

\author{Lionel Rosier}
\address{Institut Elie Cartan, UMR 7502 UdL/CNRS/INRIA,
B.P. 70239, F-54506 Vand\oe uvre-l\`es-Nancy Cedex, France}
\email{Lionel.Rosier@univ-lorraine.fr}

\begin{document}

\begin{abstract}
We prove the control and stabilization of the Benjamin-Ono equation in $L^2(\T)$, the lowest regularity where the initial value problem
is well-posed. This problem was already initiated in \cite{LinaresRosierBO}  where a stronger stabilization term was used (that makes
the equation of parabolic type in the control zone). Here we employ a more natural stabilization term related to the $L^2$ norm. 
Moreover, by proving a theorem of controllability in $L^2$, we manage to prove the global controllability in large time. Our analysis relies strongly 
on the bilinear estimates proved in \cite{MolinetPilodBO} and some new extension of these estimates established here.
\end{abstract}

\maketitle

\section{Introduction}

In this paper, we consider the BO equation posed on the periodic domain $\T =\R /(2\pi\Z)$:
\begin{equation}
\label{BOT}
u_t + {\mathcal H} u_{xx} - uu_x=0,\quad x\in\T,\ t\in\R,
\end{equation}
where  $u$ is real valued and the Hilbert transform $\mathcal H$ is defined via Fourier transform as 
\[
(\widehat{{\mathcal H}u})(\xi)=-i\, \text{sgn}(\xi)\widehat{u} (\xi),\hskip10pt \xi\in\Z.
\]

The BO equation, posed on the line,  arises as a model in waves propagation in stratified fluids (see  \cite{benjamin}, and \cite{ono}) and has widely 
been studied in many different contexts (see \cite{ABFS}, \cite{AF}, \cite{CAJT}, \cite{BP}, \cite{CoWi},
\cite{FL}, \cite{FoPo}, \cite{GFFLGP},\cite{IoKe},  \cite{Io1}, \cite{KeKo}, \cite{KoTz1},\cite{MR},\cite{MoSaTz}, \cite{MolinetBOill}, \cite{Po}, \cite{Sa}, \cite{tao}).

In the periodic setting the best  known result up to date regarding well-posedness for the Cauchy problem was obtained by Molinet \cite{molinetBOT} 
(see also \cite{MolinetPilodBO}) which guarantees the global well-posedness of the problem \eqref{BOT} for  initial $L^2$ data. 

The BO equation possesses  an infinite number of conserved quantities (see \cite{case}). The first two conserved quantities are
\begin{equation*}
I_1(t)=\int_{\T} u(x,t)\,dx
\end{equation*}
and
\begin{equation*}
I_2(t)= \int_{\T} u^2(x,t)\,dx.
\end{equation*}
Since the BO equation was derived to study the propagation of interfaces  of stratified fluids, it is
natural to think $I_1$ and $I_2$ as expressing conservation of volume (or mass) and energy, respectively.

Here we will study the equation \eqref{BOT} from a control point of view with a forcing term $f=f(x,t)$ added to the equation as a
control input:
\begin{equation}
\label{BOc}
u_t + {\mathcal H} u_{xx} -  uu_x=f(x,t),\quad x\in\T,\ t\in\R,
\end{equation}
where $f$ is assumed to be supported in a given open set $\omega \subset \T$.  In control theory the following  problems are essential: 

{\bf Exact Control Problem:} Given an initial state $u_0$ and a terminal state $u_1$ in a certain space, can one find an appropriate
control input $f$ so that the equation \eqref{BOc} admits a solution $u$ which satisfies $u(\cdot,0)=u_0$ and $u(\cdot,T)=u_1$?

{\bf Stabilization Problem:} Can one find a feedback law $f=Ku$ so that the resulting closed-loop system
\[
u_t + {\mathcal H} u_{xx} -uu_x=Ku,\quad x\in\T,\ t\in\R ^+
\]
is asymptotically stable as $t\to +\infty$?

Those questions were first investigated by Russell and Zhang in \cite{RZ} for the Korteweg-de Vries equation, which serves
as a model for propagation of surface waves along a channel:
\begin{equation}
u_t + u_{xxx} - uu_x=f,\quad x\in\T, \ t\in\R .
\label{BOf}
\end{equation}
In their work, in order to keep
the {\em mass} $I_1(t)$ conserved, the control input  was chosen of the form
\begin{equation}\label{operG}
f(x,t)=(\mathcal{G}h)(x,t):=a(x)\big( h(x,t)-\int_{\T} a(y)h(y,t)\, dy \big)
\end{equation}
where $h$ is considered as a new control input, and $a(x)$ is a given nonnegative smooth function such that
$\{x\in \T; \  a(x) >0\}=\omega$ and
$$
2\pi [a]=\int_{\T} a(x)\, dx=1.
$$


For the chosen $a$, it is easy to see that
$$
\frac{d}{dt}\int_{\T} u(x,t)\, dx=\int_{\T} f(x,t)dx=0\quad \forall t\in\R
$$
for any solution $u=u(x,t)$ of the system
\begin{equation} \label{BOG}
u_t + u_{xxx}-  uu_x=\G h,\quad x\in\T,\ t\in\R .
\end{equation}
Thus the {\em mass} of the system is indeed conserved.

The control of dispersive nonlinear waves equations on a periodic domain has been extensively studied
in the last decade: see e.g. \cite{RZ,RZ2009bis,LRZ} for the Korteweg-de Vries equation, \cite{MORZ} for the Boussinesq system,
\cite{RZ2012} for the BBM equation, and
\cite{DGL,RZ2007b,laurent,RZ2009,laurent2} for the nonlinear Schr\"odinger equation.
By contrast, the control theory of the BO equation is at its early stage.  The linear problem was studied by Linares and Ortega
 \cite{LO}. They established the following results regarding control and stabilization in this case.
\begin{TA}[\cite{LO}]\label{LOcontrol}
Let $s\ge 0$ and $T>0$ be given. Then for any $u_0,u_1\in H^s(\T )$ with $[u_0]=[u_1]$ one can find a control input $h\in L^2(0,T,H^s(\T))$ such that
the solution of the system
\begin{equation}
\label{BOlin}
u_t + {\mathcal H} u_{xx} =\G h, \qquad u(x,0)=u_0(x)
\end{equation}
satisfies $u(x,T)=u_1(x)$.
\end{TA}
 
In order to stabilize \eqref{BOlin}, they employed a simple control law
\[
h(x,t)=-\G^* u(x,t)=-\G u(x,t).
\]
The resulting closed-loop system reads\footnote{Actually, it is easy seen that $\G$ is  self-adjoint (i.e. ${\G}^* =\G$). However, we shall
keep the usual notation for the feedback $h = -{\G}^* u$  throughout .}
\[
u_t+{\mathcal H} u_{xx}=-\G\G^{*}u.
\]
\begin{TB}[\cite{LO}]\label{LOstab}
Let $s\ge 0$ be given. Then there exist some constants $C>0$ and $\lambda >0$ such that for any $u_0\in H^s(\T )$, the solution of
\[
u_t+{\mathcal H} u_{xx}=-\G\G^{*}u,\qquad u(x,0)=u_0(x),
\]
satisfies
\[
\|u(\cdot,t) - [u_0]\|  _{H^s(\T ) }  \le Ce^{-\lambda t} \|u_0-[u_0]\| _{H^s(\T ) } \qquad \forall t\ge 0.
\]
\end{TB}

One of the main difficulties to extend the linear results to the nonlinear ones comes from the  fact that one cannot use the contraction principle in its usual form 
to establish the local well-posedness of BO in $H^s_0(\T )=\{u\in H^s(\mathbb{T}),\;\;[u]=0\}$ for $s\ge 0$ (see \cite{MoSaTz}).  The method of proof in \cite{molinetBOT} and \cite{MolinetPilodBO} used strongly Tao's gauge transform, and it is not clear whether this approach
can be followed when an additional  control term is present in the equation. Nevertheless, Linares and Rosier  in \cite{LinaresRosierBO} obtained the first results 
regarding stabilization and control for the BO equation.  For completeness we will briefly describe these results.

In \cite{LinaresRosierBO}, it was considered the following feedback law
$$
h=-D(\G u),
$$
where $\widehat{D u}(\xi)=|\xi|\, \widehat{u}(\xi)$, to stabilize the BO equation.  Thus scaling in \eqref{BOf} by $u$ gives (at least formally)
\begin{equation}
\label{id}
\frac{1}{2} \|u(T)\|_{L^2(\T )}^2  + \int_0^T \|D^{\frac{1}{2}}(\G u)\|_{L^2(\T)}^2 dt
=\frac{1}{2} \|u_0\|_{L^2(\T )}^2,
\end{equation}
which suggests that the energy is dissipated over time. On the other hand,
\eqref{id} reveals a smoothing effect, at least in the region $\{ a>0\}$.
Using a {\em propagation of regularity property} in the same vein as  in \cite{DGL,laurent, laurent2,LRZ},
it was proved that the smoothing effect holds everywhere, i.e.
\begin{equation}
\label{sm}
\|u\|_{L^2(0,T;H^{\frac{1}{2}}(\T ))}\le C(T,\|u_0\|).
\end{equation}
Using this smoothing effect and the classical compactness/uniqueness argument, it was shown that
the corresponding closed-loop equation is semi-globally  exponentially stable. More precisely,
\begin{TC}[\cite{LinaresRosierBO}]\label{main1}
Let $R>0$ be given. Then there exist some constants
$C=C(R)$ and $\lambda =\lambda (R)$ such that  for any $u_0\in H^0_0(\T )$ with $\|u_0\|\le R$,
the weak solutions in the sense of vanishing viscosity of
\begin{equation}
\label{weak-dis}
u_t + {\mathcal H} u_{xx} - uu_x =-\G (D(\G u)),\qquad u(x,0)=u_0(x),
\end{equation}
satisfy
\begin{equation*}
\|u(t)\|\le Ce^{-\lambda t}\|u_0\| \qquad \forall t\ge 0.
\end{equation*}
\end{TC}

Using again the smoothing effect \eqref{sm}, it was possible to extend (at least locally) the exponential stability from $H^0_0(\T )$ to $H^s_0(\T )$ for $s>1/2$.
\begin{TD}[\cite{LinaresRosierBO}]\label{main2}
Let $s\in (\frac{1}{2}, 2]$. Then there exists $\rho >0$ such that for any data $u_0\in H^s_0(\T )$  with $\|u_0\|_{H^s(\T )}< \rho$, there exists for all $T>0$ a unique
solution $u(t)$ of \eqref{weak-dis} in the class $C([0,T],H^s_0(\T ))\cap L^2(0,T, H^{s+\frac{1}{2}}_0(\T ))$. Furthermore, there exist some constants $C>0$ and
$\lambda >0$
such that
\begin{equation*}
\|u(t)\| _s \le Ce^{-\lambda t} \|u_0\|_s\qquad \forall t\ge 0.
\end{equation*}
\end{TD}

Finally, including the same feedback law $ h=  -D(\G u)$ in the control input to obtain a smoothing effect, it was derived
an exact controllability result for the full equation as well. More precisely,
\begin{TE}[\cite{LinaresRosierBO}]\label{main3}
Let $s\in (\frac{1}{2}, 2] $ and $T>0$ be given. Then there exists $\delta>0$ such that for any $u_0,u_1\in H^s_0(\T)$ satisfying
\begin{equation*}
\|u_0\|_{H^s(\T)} \le \delta,\quad \|u_1\|_{H^s(\T )}\le \delta
\end{equation*}
one can find a control input $h\in L^2(0,T,H^{s-\frac{1}{2}} (\T))$ such that the system \eqref{BOG} admits a solution
$u\in C([0,T],H^s_0(\T))\cap L^2(0,T,H^{s+\frac{1}{2}}_0(\T ) )$ satisfying
\begin{equation*}
u(x,0)=u_0(x),\quad u(x,T)=u_1(x).
\end{equation*}
\end{TE}

Our main purpose here is to obtain the control and stabilization result in the largest space where local well-posedness is known for
the BO equation, that is,  for initial data in $L^2(\mathbb T)$.  As we commented above the main difficulty to treat the problem for data
in $L^2$ was to use the gauge transform $w=-\frac{i}{2}P_+(ue^{-\frac{i}{2}F})$ introduced in \cite{tao}.  We first consider the Cauchy problem  associated to the BO 
equation with a forcing term,
that is,
\begin{equation}\label{BOforcing}
\begin{cases}
u_t + {\mathcal H} u_{xx} -uu_x=g, \quad x\in\T,\;t>0,\\
u(x,0)=u_0(x)
\end{cases}
\end{equation}
for data in $L^2(\mathbb T)$ and $g\in L^2([0,T],L^2(\T))$. Then, we establish the  well-posedness theory for \eqref{BOforcing}.
The main tool used is the bilinear estimates approach
employed by Molinet and Pilod \cite{MolinetPilodBO} to show the local well-posedness of the BO equation in both cases in
the real line and in the periodic setting for data in $L^2(\mathbb T)$ (see \cite{IoKe} for the first proof of this result in $L^2(\R)$).  The argument in
 \cite{MolinetPilodBO}  takes advantage  of uniform estimates for small data  for which  the gauge transform is well defined and a
 scaled argument (subcritical) to consider any size data.  In our case, we apply some Sobolev estimates  which avoid the use
 of uniform estimates and we apply directly the argument in \cite{MolinetPilodBO}.  Then we have to analyze the gauge transform
 for large frequency data,  for which this transformation has a bad behavior. Doing so we extend  previous estimates  to overcome
 this difficulty. For instance we refine the bilinear estimates and obtain an extra time factor  (see Lemma \ref{bilinestim}) used  to
 deal with large data as well as  intermediate estimates to treat the same situation (see the Appendix).

Next we describe the main results in  this paper.


We begin with the exact controllability result. More precisely,

\begin{theorem}\label{control}\hskip10pt
\begin{itemize}
\item[(i)]  (Small data) For any $T>0$, there exists some $\delta>0$   such that for any $u_0, u_1\in L^2(\T)$ with
\begin{equation*}
\|u_0\|\le \delta,\;\; \|u_1\|\le \delta \hskip10pt\text{and}\hskip10pt [u_0]=[u_1]
\end{equation*}
one can find a control input $h\in L^2([0,T],L^2(\T))$ such that the solution $u$ of the system
\begin{equation}\label{BOnonlin}
\begin{cases}
u_t + {\mathcal H} u_{xx} -uu_x=\G h,\\
u(x,0)=u_0(x)
\end{cases}
\end{equation}
satisfies $u(x,T)=u_1(x)$ on $\T$.\\
\item[(ii)] (Large data) For any $R>0$, there exists a  positive $T=T(R)$ such that the above property holds for any $u_0, u_1\in L^2(\T)$   with
\begin{equation*}
\|u_0\|\le R,\;\; \|u_1\|\le R \hskip10pt\text{and}\hskip10pt [u_0]=[u_1].
\end{equation*}
\end{itemize}
\end{theorem}

\begin{remark} \label{rkchgevar} We will give the proof of the theorem assuming that $[u_0]=[u_1]=0$. The general case can be done
after the change of variables $\widetilde{u}(x,t)=u(x-t[u_0],t)-[u_0]$ is made in the equation in \eqref{BOnonlin}. The
results in Theorem A remain valid for the new equation in $\widetilde{u}$. 
\end{remark}

The argument of proof is as follows. The control result for large data (ii) will be a combination of a stabilization result presented below and of the result (i) for small data, as is usual in control theory (see for instance \cite{DLZ,DGL,laurent,laurent2,LRZ}). The local control result (ii) will be proved using a perturbation argument from the linear result \cite{LO}. The difficulty comes from the fact that we need some estimates at higher order for the nonlinear equation, which are quite complicated due to the absence of a direct Duhamel formulation. Indeed, we need to have an expansion for all the elements used for the well-posedness theory. These estimates will come from the Benjamin-Ono equation verified by $u$, from the equation verified by the gauge transform $w$ and from  the \lq\lq inversion" of the gauge transform. For example, we will need  some estimates of the form $2iw=P_+ u+\petito{\nor{(u,v)}{X}^2}$ and some Lipschitz type estimates of this fact.

In order to stabilize \eqref{BOnonlin}, we employ a simple control law
\begin{equation*}
h(x,t)=-\G^* u(x,t).
\end{equation*}
where $\G^*$ denotes the adjoint operator of $\G$.
Thus the resulting closed-loop system is
\begin{equation*}
u_t+{\mathcal H} u_{xx}+uu_x=-\G\G^{*}u.
\end{equation*}

\begin{theorem}\label{stabilization}
There exist some nondcreasing functions $c,\;\lambda: \R_{+}\to \R_{+}$ such that for any $u_0\in L^2(\T )$, the solution of
\begin{equation}\label{stabilization1}
\begin{cases}
u_t+{\mathcal H} u_{xx}-uu_x=-\G\G^{*}u,\\
u(x,0)=u_0(x)
\end{cases}
\end{equation}
satisfies
\begin{equation*}
\|u(\cdot,t) -[u_0]\|  _{L^2(\T ) }  \le Ce^{-\lambda t} \|u_0-[u_0]\| _{L^2(\T ) } \qquad \forall t\ge 0
\end{equation*}
where $\lambda=\lambda(\|u_0\|)$, $C=C(\|u_0\|)$.
\end{theorem}

\begin{remark} \label{rkchgevarstab}After the change of variables $\widetilde{u}(x,t)=u(x-t[u_0],t)-[u_0]$ described in Remark \ref{rkchgevar}, we are left with some solutions 
with $[\widetilde{u}]=0$ and with the same equation but with $\G$ replaced by $\G_{\mu}$ defined by
\begin{equation}\label{operGnu}
(\mathcal{G_{\mu}}h)(x,t):=a(x-\mu t)\big( h(x-\mu t,t)-\int_{\T} a(y)h(y,t)\, dy \big)
\end{equation}
with $\mu=[u_0]$.

We will give the proof of the theorem assuming that $\mu=[u_0]=0$ and detail in some remarks the necessary modifications when $\G$ is replaced by $\G_{\mu}$. 
\end{remark}

The proof of the stabilization theorem will be obtained by an observability estimate proved by the compactness-uniqueness method. The difference from the case of the nonlinear Schr\"odinger equation (see for instance \cite{DGL,laurent,laurent2}) or other nonlinear conservative equations (\cite{DLZ,LRZ,PVZ,P}) comes again from the use of the gauge transform. We need to propagate some information from the zone of damping to the whole space. Yet, the absence of Duhamel formulation prevents from doing that directly on the solution $u$. The idea is to transfer this information from the solution $u$ to the gauge transform $w$, next  to apply some propagation results to $w$ and finally to obtain the expected result by returning to the original solution $u$.

As it was already  mentioned, one of our main ingredients in our analysis is the study of  the well-posedness of the IVP
\begin{equation}\label{sourceivp}
\begin{cases}
\partial_t u+{\mathcal H}\partial_x^2 u- u\partial_xu=g, \quad x\in\mathbb{T}, \;t>0,\\
u(x,0)=u_0(x),
\end{cases}
\end{equation}
for data $u_0\in L^2(\T)$ and where $g\in L^2([0,T], L^2(\mathbb{T}))$.

More precisely,

\begin{theorem}\label{sourcelwp}
For any  $u_0\in L^2(\T)$, $g\in L^2([0,T],L^2(\T))$  and $T>0$, there exists a solution 
\begin{equation}\label{sourcelwp1}
u\in C([0,T], L^2(\T))\cap X_T^{-1,1}\cap L^4_TL^4(\T)
\end{equation}
of \eqref{sourceivp} such that
\begin{equation}\label{sourcelwp2}
w=\partial_xP_{+}(e^{-\frac{i}2\partial_x^{-1}\tilde{u}})\in X_2
\end{equation}
where 
\begin{equation*}
\tilde{u}=u(x-t[u_0], t)-[u_0] \text{\hskip10pt and\hskip10pt} \widehat{\partial_x^{-1}}:=\frac{1}{i\xi},\;\xi\in\Z^{*}.
\end{equation*}

The solution $u$ is unique in the class
\begin{equation*}
u\in L^{\infty}((0,T); L^2(\T))\cap L^4((0,T)\times \T) \hskip10pt and \hskip10pt w\in X^{0,1/2}_T.
\end{equation*}

Moreover, $u\in C([0,T];L^2(\T))$ and the flow map  data--solution: $(u_0,g)\to u$ is locally Lipschitz continuous from
$L^2(\T)$  into $C([0,T];L^2(\T))$.
\end{theorem}

The proof of this theorem does not follow directly from the theory established by Molinet and Pilod \cite{MolinetPilodBO}.
We have to distinguish two cases, small and large data instead.  To apply the scale argument in \cite{MolinetPilodBO} there is a need of
some uniform estimates that do not hold when we introduced the source term $g$ and consider any data.  
To deal with large data we have to go around the gauge transform since in this case it not \lq\lq invertible".  We prove the
theorem when $[u_0]=0$. The general case follows as we commented above.

The plan of the paper is as follows: in Section 2 we introduce some notations  and establish estimates needed in our analysis. 
General estimates which do not depend on the size of the data are derived in Section 3. In Section 4 we deduce estimates used to complete the proof of  Theorem \ref{sourcelwp}
for small data. Next in Section 5 we establish Theorem \ref{sourcelwp} for large data.  
The control results are proved in Section 6.  Section 7 contains
preliminary tools to establish Theorem \ref{stabilization}. 
The stabilization is demonstrated  in Section 8, and finally we include
an appendix containing some technical results.

\section{Notations and Preliminary Estimates}

We use standard  notation in Partial Differential Equations. In addition, we will use $C$  to denote various constants that may change 
from line to line. For any positive numbers  $a$ and $b$, we use the notation $a\lesssim b$ to mean that there exists a constant $C$
such that $a\le C\,b$.

For a $2\pi$-periodic function $f$ we define its Fourier transform on $\Z$  by
$$
\widehat{f}(\xi)=\int_{\T} e^{-ix\xi}\,f(x)\,dx,\quad \xi\in\Z.
$$

The free group associated with the linearized Benjamin-Ono equation denoted by $\U(\cdot)$ is defined as
\begin{equation*}
\widehat{\U(t)f}(\xi)=e^{it|\xi|\xi}\widehat{f}(\xi).
\end{equation*}

The norm of Sobolev spaces $H^s(\T)$ will be denoted by $\|\cdot\|_s$ and when $s=0$  the notation $\|\cdot\|$ will be used.

We  will use the following projection operators:  for $N\in \N^*$  we define
\begin{equation}\label{projections}
\begin{split}
&\widehat{P_{\pm}f}=\chi_{\R_{\pm}}\widehat{f},\hskip35pt\widehat{P_{\leq 0}f}=\chi_{\{\xi \leq 0\}}\widehat{f}, \;\;\;\;
\widehat{P_{\ge N}f}=\chi_{\{\xi \geq N\}}\widehat{f}\\
&\widehat{P_{\le N}f}=\chi_{\{\xi \leq N\}}\widehat{f},\hskip15pt \widehat{P_0 f}=\widehat{f}(0),\hskip35pt
\widehat{P_{ N}f}=\chi_{\{|\xi| \leq N\}}\widehat{f},
\end{split}
\end{equation}
and
\begin{equation}
\widehat{Q_{N}f}=\chi_{\{|\xi| > N\}}\widehat{f}.
\end{equation}

For $s, b\in \R$ we define the spaces  $X^{s,b}$, $Z^{s,b}$ and $\widetilde{Z}^{s,b}$ via the norms
\begin{equation*}
\|v\|_{X^{s,b}}=\Big(\int\int \langle \tau+|\xi|\xi\rangle^{2b}\langle \xi \rangle^{2s}\,|\widehat{v}(\xi,\tau)|^2\,d\xi d\tau\Big)^{1/2}\end{equation*}
\begin{equation*}
\|v\|_{Z^{s,b}}=\Big(\int(\int \langle \tau+|\xi|\xi\rangle^{b}\langle \xi \rangle^{s}\,|\widehat{v}(\xi,\tau)|\,d\xi)^2 d\tau\Big)^{1/2},
\end{equation*}
and 
\begin{equation*}
\|v\|_{\widetilde{Z}^{s,b}}=\|P_0 v\|_{Z^{s,b}}+\Big(\underset{N}{\sum}\| v_N\|_{Z^{s,b}}^2\Big)^{1/2},
\end{equation*}
where  $d\xi$ denotes the counting measure in $\Z$, $\;\langle x\rangle= 1+|x|$, and $v_N$ corresponds to a classical Littlewood-Paley decomposition (see \cite{MolinetPilodBO}).

We also define the spaces $X^s_2$ and ${X^s_2}^{'}$ as
\begin{equation*}
\|v\|_{X_2^s}=\|v\|_{X^{s,1/2}}+\|v\|_{\widetilde{Z}^{s,0}}
\end{equation*}
and
\begin{equation*}
\|v\|_{{X_2^s}^{'}}=\|v\|_{X^{s,-1/2}}+\|v\|_{\widetilde{Z}^{s,-1}}.
\end{equation*}

When $s=0$ we will use $X_2$ and ${X_2'}$ to simplify the notation.

Next we define the restriction in time spaces. More precisely, for any function space $B$ and any $T>0$, we denote by $B_T$ the 
corresponding restriction in time space endowed with the norm
\begin{equation}\label{restrictionTS}
\|u\|_{B_T} = \inf_{v\in B}\,\big\{\|v\|_B\; :\; v(\cdot)=u(\cdot) \;\;\text{in}\;\; [0,T)\big\}.
\end{equation}

Note that, as pointed out in \cite{MolinetPilodBO}, we have the embedding
\bna
X_{2,T}^s \subset \widetilde{Z}^{s,0}_T \subset C([0,T],H^s(\T)).
\ena
We end this section by introducing the functional space $X$ with the norm
\bna
\|(u,w)\|_{X}=\|u\|_{L^{\infty}([0,T],L^2)}+\|u\|_{L^{4}([0,T]\times \T)}+\|u\|_{X_T^{-1,1}}+\|w\|_{X_T^{0,1/2}}+\|w\|_{\widetilde{Z}_T^{0,0}}
\ena
and we denote by $X_1$ the functional space  corresponding to the norm for $u$ and $X_2$ the functional space given by the norm  corresponding to $w$ as above. 
We will also use the notation  $U=(u,w)\in X$.\\

\vskip.3cm

\subsection{Linear estimates}

We will begin presenting some estimates in the Bourgain spaces. The reader is referred  to \cite{ginibreTV} for details.

\begin{lemme}\label{lema1} Let $s\in \R$. Then
\begin{equation*}
\|\eta(t)\U(t)h\|_{X_2^s}\lesssim \|h\|_{H^s}.
\end{equation*}
\end{lemme}

\begin{lemme}\label{lemma1b}
Let $s\in\R$. Then for any $0<\delta<1/2$,
\begin{equation}\label{lemma1bA}
\|\eta(t)\int_0^t \U(t-t')H(t')\,dt'\|_{X^{s,\frac12+\delta}}\lesssim \|H\|_{X^{s,-\frac12+\delta}}
\end{equation}
and
\begin{equation}\label{lemma1bB}
\|\eta(t)\int_0^t \U(t-t')H(t')\,dt'\|_{X_2^{s}}\lesssim \|H\|_{X^{s,-\frac12}}+\|H\|_{\widetilde{Z}^{s,-1}}.
\end{equation}
\end{lemme}

\begin{lemme}\label{le1}
 For any $T>0$, $s\in\R$ and  $-\frac{1}{2} <b'\leq b<\frac{1}{2}$, it holds
\begin{equation}
\|u\|_{X^{0,b'}_T}\lesssim T^{b-b'}\|u\|_{X^{0,b}_T}.
\end{equation}
\end{lemme}

The next estimate is a Strichartz type estimate (see \cite{MolinetPilodBO}).
\begin{lemme}\label{strichartz}
It holds that
\begin{equation*}
\|u\|_{L^4_{x,t}}\lesssim \|u\|_{X^{0,\frac38}}
\end{equation*}
and for any $T>0$ and $\frac38\le b\le \frac12$,
\begin{equation*}
\|u\|_{L^4_{x,T}}\lesssim T^{b-\frac38}\|u\|_{X^{0,b}_T}.
\end{equation*}
\end{lemme}

Recalling the notation ${X_2'}=X^{0,-1/2}\cap {\widetilde{Z}}^{0,-1}$ we have that
\begin{lemme}\label{lemma-estimaYL2} For $0<\epsilon\ll 1$, it holds
\begin{equation}\label{estimaYL2}
\|z\|_{X_2'}\lesssim T^{1/2-\e}\|z\|_{L^2([0,T]\times \T)}.
\end{equation}
\end{lemme}

\begin{proof}
From the definition we have that
\begin{equation*}
\begin{split}
\|z\|_{Z^{0,-1}}^2&=\int_{\xi}\big(\int_{\tau}\big\langle \tau+|\xi|\xi\big\rangle^{-1}|\hat{z}(\tau,\xi)|\big)^2
=\int_{\xi}\big(\int_{\sigma}\big\langle \sigma\big\rangle^{-1}|\hat{z}(\sigma -|\xi|\xi,\xi)|\big)^2\\
&\lesssim \int_{\xi}\int_{\sigma}\big\langle \sigma\big\rangle^{-1+2\e}|\hat{z}(\sigma -|\xi|\xi,\xi)|^2\lesssim \int_{\xi}\int_{\tau}\big\langle \tau
+|\xi|\xi\big\rangle^{-1+2\e}|\hat{z}(\tau,\xi)|^2=\|z\|_{X^{0,-1/2+\e}}^2
\end{split}
\end{equation*}
where we have used the Cauchy-Schwarz inequality and the integrability of $\big\langle \sigma\big\rangle^{-1-2\e}$. We can use Lemma \ref{le1} to finish the proof.
\end{proof}

Define $F=\partial_x^{-1}u$ where $\widehat{\partial_x^{-1} u}=\xi^{-1}\widehat{u}$, $\xi\in \Z^{*}$.
\begin{lemme} \label{estimlipsexp}Let  $F_i$, $i=1,2$,  be defined as above. Then
\begin{equation}\label{estimlipsexp1}
\|e^{\frac{i}{2}F_1}-e^{\frac{i}{2}F_2}\|_{L^{\infty}(\T)}\lesssim \|u_1-u_2\|_{L^2},
\end{equation}
\begin{equation}\label{estimlipsexp2}
\|P_+\big(e^{\frac{i}{2}F_1}-e^{\frac{i}{2}F_2}\big)\|_{L^{\infty}(\T)}\lesssim  \big(1+\|u_1\|_{L^2}+\|u_2\|_{L^2}\big)\|u_1-u_2\|_{L^2},
\end{equation}
and
\begin{equation}\label{estimlipsexp3}
\|P_+\big(e^{\frac{i}{2}F_1}-1\big)\|_{L^{\infty}(\T)}\lesssim \|u_1\|_{L^{2}}.
\end{equation}
\end{lemme}

\begin{proof}
Using Sobolev embedding and the definition of $F$ we have
$$
\|e^{\frac{i}{2}F_1}-e^{\frac{i}{2}F_2}\|_{L^{\infty}(\T)}\leq \|F_1-F_2\|_{L^{\infty}}\leq C\,\|F_1-F_2\|_{H^1}\leq C\,\|u_1-u_2\|_{L^2}.
$$

Next we prove \eqref{estimlipsexp2}. Since $P_{+}$ is not a continuous operator in $L^{\infty}$ we use first the Sobolev embedding and then the $L^2$ continuity.
\begin{equation*}
\|P_+\big(e^{\frac{i}{2}F_1}-e^{\frac{i}{2}F_2}\big)\|_{L^{\infty}(\T)}\leq C\,\|e^{\frac{i}{2}F_1}-e^{\frac{i}{2}F_2}\|_{H^{1}(\T)} \leq C\,\big(1+\|F_1\|_{H^1}
+\|F_2\|_{H^1}\big)\|F_1-F_2\|_{H^1}.
\end{equation*}
The definition of $F_i$ yields the result.

Similarly, we obtain
\begin{equation*}
\|P_+\big(e^{\frac{i}{2}F_1}-1\big)\|_{L^{\infty}(\T)}\leq C\,\|e^{\frac{i}{2}F_1}-1\|_{H^{1}(\T)} \leq C\,\|F_1\|_{H^1}\leq C\,\|u_1\|_{L^{2}}.
\end{equation*}
\end{proof}

The next result guarantees the existence of smooth solutions for the Cauchy problem  \eqref{BOnonlin}.

\begin{theorem}[Existence of smooth solutions] \label{smoothsol}\hskip2pt
\begin{itemize}
\item (Local) For a given $T>0$ and $s>3/2$, the initial value problem
\begin{equation}\label{bosource}
\begin{cases}
u_t+\H u_{xx} -uu_x=g, \qquad x\in\T,\; t>0,\\
u(x,0)=u_0(x)
\end{cases}
\end{equation}
is locally well-posed for initial data $u_0\in H^s(\T)$ and $g\in L^2([0,T],  H^s(\T))$. Moreover, the map $(u_0,g) \to u$ is continuous
from $H^s(\T)\times L^2([0,T], H^s(\T)) \to C([0,T], H^s(\T))$.
\item  (Global) Let $s\ge 2$, $T>0$. Then for any $(u_0,g)\in H^s(\T)\times L^2([0,T], H^s(\T))$ the IVP \eqref{bosource} has 
a unique solution
\begin{equation*}
u\in C([0,T], H^s(\T)).
\end{equation*}
\end{itemize}
\end{theorem}
We notice that the result is also true with a damping term $-\G\G^{*}u$ instead of $g$. 
\begin{proof} The local theory follows using parabolic regularization (see for instance \cite{Ioriobook}) for the existence and uniqueness. 
The continuous dependence  follows by using the Bona-Smith argument \cite{BonaSmith}.  

To extend the local theory globally we will make use of some quantities conserved by the BO flow. We will sketch the proof for completeness.
We begin with the $L^2$ estimate. We proceed as follows.
Multiply the equation in \eqref{bosource} by $u$ and integrate with
respect to $x$.
\begin{equation}\label{3.3}
\frac12\frac{d}{dt}\,\int\,u^2\,dx =\int\,u g\,dx\le \frac12\|u\|^2+\|g(\cdot,t)\|^2.
\end{equation}
Gronwall's inequality implies
\begin{equation}\label{3.5}
\|u(t)\|^2\le C\,\{\|u_0\|^2+\int_0^T\,\|g(t)\|_{L^2}^2\,dt\}
\,\exp(C\,T), \quad\text{for}\quad 0\le t \le T.
\end{equation}

Next we obtain a $H^1$ {\it a priori} estimate. We use the fourth conserved quantity associated to the BO flow, that is,
$$
\Psi_4(u)=\int\,\Big( 2(\partial_x u)^2-\frac32u^2\,\mathcal H(\partial_x u)+\frac{u^4}4\Big)\,dx.
$$

We apply the operator $4\partial_xu\partial_x-3u\mathcal H(\partial_x u)-\frac32\,u^2\mathcal H\partial_x +u^3$ to the
equation in \eqref{bosource} and integrate by parts to obtain
\begin{equation}\label{3.6}
\begin{split}
\frac{d}{dt}\int\,\Big(&2(\partial_x u)^2-\frac32 u^2\,\mathcal H(\partial_x  u)+\frac{u^4}4\Big) \,dx\\
&=\int (4\partial_xu\partial_xg-3u\mathcal H(\partial_x u)g-\frac32\,u^2\mathcal H\partial_xg +u^3g)\,dx
\end{split}
\end{equation}

The Cauchy-Schwarz inequality, interpolation  and Hilbert transform's properties give
\begin{equation}\label{3.6b}
\begin{split}
&\Big| \int \Big(4\partial_xu\partial_xg-3u\mathcal H(\partial_x u)g-\frac32\,u^2\mathcal H\partial_xg +u^3g\Big)\,dx\Big|\\
&\le C\,\|\partial_x u\|^2+c\,\big(\|u\|^2(\|g\|\|\partial_xg\|+\|\partial_x g\|^{4/3})+\|u\|^4\|g\|^2+\|\partial_x g\|^2\big).
\end{split}
\end{equation}

Combining \eqref{3.6} and \eqref{3.6b} we obtain
\begin{equation}\label{3.9}
\frac{d}{dt}\Psi_4(u(t)) 
\le  C\,h(\|u(t)\|,\|g(t)\|_1)
+C\,\|\partial_x u(t)\|^2,
\end{equation}
where
\begin{equation}\label{3.10}
h(\|u(t)\|,\|g(t)\|_1)=\|u(t)\|^2(\|g(t)\|^2_1+\|g(t)\|^{4/3}_1)+\|u(t)\|^4\|g(t)\|^2+\|g(t)\|^2_1.
\end{equation}

 By using the Cauchy-Schwarz inequality, Sobolev embedding and  Gagliardo-Nirenberg interpolation we also deduce that
\begin{equation}\label{3.11}
\Psi_4(u(0))\le C(\|u_0\|_1).
\end{equation}

Integrating \eqref{3.9} with respect to $t$ and using \eqref{3.11} it follows then that
\begin{equation}\label{3.12}
\Psi_4(u(t))\le 
C(\|u_0\|_1)+H(t)+C\,\int_0^t\,\|\partial_x u(s)\|^2\,ds
\end{equation}
where
$$
H(t)=\int_0^t\,h(\|u(s)\|,\|g(s)\|_1)\,ds\le \widetilde{h}(t, \sup_t\|u(t)\|, \|g\|_{L^2([0,t], H^1(\T))}).
$$

On the other hand, 
$$
\|\partial_x u(t)\|^2= \frac12\Psi_4(u(t))+\frac34\int\,u(t)^2\,\mathcal H(\partial_xu(t))\,dx-\frac18\int\,u(t)^4\,dx.
$$

Interpolation and Young's inequality imply then
\begin{equation}\label{3.13}
\|\partial_x u(t)\|^2 \le \frac12\Psi_4(u(t))+(\eta_1+\eta_2)\|\partial_xu(t)\|^2 +c(\eta_1, \eta_2)\|u(t)\|^6,
\end{equation}
where $\eta_1,\,\eta_2$ are positive numbers. 

A suitable choice of $\eta_1,\, \eta_2$ in \eqref{3.11} combined with inequality 
\eqref{3.12} implies
$$
\|\partial_x u(t)\|^2\le 
C(\|u_0\|_1)+H_1(t)
+C\int_0^t\,\|\partial_xu(t')\|^2\,dt',
$$
where $H_1(t)=H(t)+C\|u(t)\|^6$.

Gronwall's inequality implies
\begin{equation}\label{3.14}
\|\partial_x u(t)\|^2\le
\big\{C(\|u_0\|_1)+\int_0^T\,H_1(t')\,dt'\big\}
\,\exp\{CT\}, \quad\text{for}\quad 0\le t \le T.
\end{equation}

This inequality and  the estimate \eqref{3.5} for $\|u\|$ give an {\it a priori} estimate for the $H^1$-norm.

\vskip 10pt
To establish a $H^2$ {\it a priori} estimate for solutions of \eqref{bosource} we consider the sixth conserved quantity associated to the Benjamin-Ono equation, that is,
\begin{equation*}
\begin{split}
\Psi_6(u)&=\int\,\biggl[\frac{u^6}6 -\biggl \{\frac54u^4\,\mathcal H(\partial_x u)
+\frac53 u^3 \,\mathcal H(u\partial_x u)\biggr\}\biggr]\,dx\\
&+\frac52 \int\,\biggl[5u^2(\partial_x u)^2 +u^2\{\mathcal H(\partial_x u)\}^2 
+2u\,\mathcal H(\partial_x u)\,\mathcal H(u\partial_x u) \biggr]\,dx\\
&+10\int\, \biggl[ (\partial_x u)^2\,\mathcal H(\partial_x u) 
+2u\,\partial_x^2u\,\mathcal H(\partial_x u)\biggr]\,dx+8\int \,(\partial_x^2 u)^2\,dx.
\end{split}
\end{equation*}

The argument to show the {\it a priori} estimate in this case
although similar is rather technical, so we will give the
final statement of the results without giving the details.
Proceeding as in the previous case we can show that
\begin{equation}\label{3.15}
\Phi_6(u(t))\le C(\|u_0\|_{2})+F_0(t)+ C\int_0^t\|\partial_x^2u(t')\|^2\,dt',
\end{equation}
where
$$
F_0(t)=\int_0^t\,f(\|u(s)\|,\|\partial_x u(s)\|,\|g(s)\|_2)\,ds.
$$

Using the Cauchy-Schwarz inequality, Gagliardo-Nirenberg's interpolation,
Young's inequality and \eqref{3.15} we get
\begin{equation*}
\begin{split}
\|\partial_x^2 u(t)\|^2&\le \frac18\Psi_6(u)+F_1(\|u(t)\|,\|\partial_x u(t)\|)\\
&\le C(\|u_0\|_{2})+F_0( t)+F_1(\|u(t)\|_1)+ C\int_0^t\|\partial_x^2 u(t')\|^2\,dt'.
\end{split}
\end{equation*}
Thus an application of Gronwall's inequality gives
\begin{equation}\label{3.16}
\|\partial_x^2\,u(t)\|^2\le \{C(\|u_0\|_{2})+\int_0^t\,(F_0(t')+F_1(\|u(t')\|_1))\,dt'\}\,\exp(Ct).
\end{equation}

Once the $H^2$ {\it a  priori } estimates are available, the structure of the 
equation and the Sobolev embedding theorem allow us to obtain an {\it a priori} estimate in higher order Sobolev spaces. This is 
basically the content of the next claim.

\noindent \textbf{Claim}.
For any $s\ge 2$ and $0\le t\le T$ the solutions of the IVP \eqref{bosource}  satisfy
\begin{equation}\label{3.17}
\begin{split}
\|u(t)\|_s\le \,C \{\|u_0\|_s+\,\int_0^T\,\|g(t')\|_s\,dt'\}\,\exp\{C\,\int_0^T\,(1+\|u(t')\|_2)\,dt'\big\}.
\end{split}
\end{equation}

\begin{proof}[Proof of the Claim]
Apply the operator $J^s$ to the equation \eqref{bosource}, 
then multiply by $J^su$ and integrate with respect to $x$:
\begin{equation}\label{4.0}
\begin{split}
\frac12\frac{d}{dt}\int (J^su)^2\;dx&=\int (J^s(u\partial_xu )\,J^su)\,dx +\int J^su\,J^sg\,dx\\
&\le | \int (J^s(u\partial_xu )\,J^su)\,dx | +\|u\|_s\|g\|_s
\end{split}
\end{equation}
where in the last inequality we use the Cauchy-Schwarz inequality. To estimate the first term on the right hand side of the last inequality we use the periodic version
of  Kato-Ponce commutator estimate (see \cite{IoKe2}, \cite{KaPo}) 

\begin{equation}\label{4.1}
\begin{split}
\Big|\int J^s(u\partial_xu )\,J^su\,dx\Big| &\le \big|\int [J^s,u]\partial_x uJ^su\,dx\big|+| \int\partial_xu (J^su)^2\, dx\big|\\
&\le C(\|u\|_{L^{\infty}}+\|\partial_xu\|_{L^{\infty}})\|u(t)\|_s^2.
\end{split}
\end{equation}

Combining \eqref{4.0}, \eqref{4.1}  and Young's inequality we obtain
\begin{equation}\label{4.2}
\frac12\frac{d}{dt}\|u(t)\|_s^2\le C(\|u\|_{L^{\infty}}+\|\partial_xu\|_{L^{\infty}})\|u(t)\|_s^2+\|g(t)\|_s^2.
\end{equation}

Hence Sobolev's embedding and Gronwall's inequality yield \eqref{3.17} proving our claim.
\end{proof}

Some estimates similar to all the  previous ones, including  the {\it a priori} estimate \eqref{3.17}, are valid if we add the term $\epsilon\partial_x^2 u$, $\epsilon>0$, on the right hand side of \eqref{bosource}. We refer to Iorio \cite{Io1} Lemma B1 and Lemma B3 for more details (note that in this reference, the sign of the nonlinearity is different but we can easily get to the same equation by the change of unknown $u\leftrightarrow -u$ at the cost of a slight change of the conserved quantities and the sign of their coefficients of odd order).  Since the local theory is established via the parabolic regularization  method the {\it a priori} estimates above can be used to extend the 
local solution globally. This completes the proof of Theorem \ref{smoothsol}
\end{proof}

\begin{lemme}
\label{continuiteXt}
Let $u\in X$ then the function
\begin{equation}
\begin{cases}
f: (0,T] \to \;\R\\
\hskip31pt t\;\; \mapsto\; \|u\|_{X_t}
\end{cases}
\end{equation}
is continuous, where $X_t$ is the norm of restriction on $[0,t)$.
Moreover, there exists $C$ such that
\begin{equation*}
\lim_{t\rightarrow 0} f(t) \leq C \|(u(0),w(0))\|_{L^2}.
\end{equation*}
 \end{lemme}

\begin{proof} For a proof of this result see for instance \cite{laurent}  (Lemma 1.4)  or  \cite{ken-p}  (Lemma 6.3).
\end{proof}

In what follows, we will omit the dependence on $T$ of the space $X$ and we will denote $X=X_T$ for simplicity.

\subsection{The equation after gauge transform}

Now we define the gauge transformation to obtain a new equation.  Consider the IVP
\begin{equation}\label{initial0}
\begin{cases}
\partial_t u+\H \partial_x^2 u=u\partial_x u+g=\frac{1}{2}\partial_x (u^2)+g,\\
u(0)=u_0
\end{cases}
\end{equation}
where $u$ a real valued function and  $u_0$, $g$ are real valued functions with zero mean.  Set $w=\partial_x P_+(e^{-\frac{i}{2}F})$, $F=\partial^{-1}_x u$, and 
$G=\partial_x^{-1}g$, then \eqref{initial0} becomes
\begin{equation}\label{initialGT}
\begin{cases}
\partial_t F+\H \partial_x^2 F=\frac{1}{2} (\partial_x F)^2-\frac{1}{2}P_0 (F_x^2)+\partial_x^{-1}g\\
F(0)=\partial_x^{-1} u_0.
\end{cases}
\end{equation}

Denote by $W=P_+(e^{\frac{i}{2}F})$,  the solution of
\begin{equation*}
\partial_t W-i\partial_x^2 W=-P_+\big[( \partial_x^2 P_- F) e^{-\frac{i}{2}F}\big]-\frac{i}{2}P_+\big[\big(-\frac{1}{2}P_0 (F_x^2)+G\big) e^{-\frac{i}{2}F}\big].
\end{equation*}
Since $e^{-\frac{i}{2}F}=W+P_{\leq 0}e^{-\frac{i}{2}F}$ and the cancelation $P_+\big[( \partial_x^2 P_- F) P_{\leq 0}e^{-\frac{i}{2}F}\big]=0$ we obtain
\begin{equation*}
\partial_t W-i\partial_x^2 W=-P_+\big[( \partial_x^2 P_- F) W\big]-\frac{i}{2}P_+\big[\big(-\frac{1}{2}P_0 (F_x^2)+G\big) e^{-\frac{i}{2}F}\big].
\end{equation*}
Denoting $w=\partial_x W=-\frac{i}{2}P_+(ue^{-\frac{i}{2}F})$, it follows that
\begin{equation*}
\begin{split}
\partial_t w-i\partial_x^2 w&=-\partial_x P_+\big[W(  P_- \partial_x u) \big]-\frac{i}{2}\partial_xP_+\big[\big(-\frac{1}{2}P_0 (F_x^2)+G\big) e^{-\frac{i}{2}F}\big]\\
&=-\partial_x P_+\big[W(  P_- \partial_x u) \big]-\frac{i}{2}P_+\big[g e^{-\frac{i}{2}F}\big]-\frac{1}{4}P_+\big[\big(-\frac{1}{2}P_0 (u^2)+G\big) ue^{-\frac{i}{2}F}\big].
\end{split}
\end{equation*}

We want to prove that at order one, if $u(0)=u_0$,  $w=-\frac{i}{2}P_+u$ with
\begin{equation*}
u(t)=e^{t\H\partial_x^2}u_0+\int_0^t e^{(t-\tau)\H\partial_x^2}g(\tau)\,d\tau+\grando{\nor{u_0}{L^2}^2+\|g\|_{L^2([0,T],L^2)}^2}.
\end{equation*}

The  main tool in our analysis is equation
\begin{equation}\label{equaw}
\begin{cases}
\partial_t w-i\partial_x^2 w=-\partial_x P_+\big[W(  P_- \partial_x u) \big]-\frac{i}{2}P_+\big[g e^{-\frac{i}{2}F}\big]-\frac{1}{4}P_+\big[\big(-\frac{1}{2}P_0 (u^2)
+G\big) ue^{-\frac{i}{2}F}\big]\\
w(0)=w_0.
\end{cases}
\end{equation}
We denote the right hand side of \eqref{equaw} as
\begin{equation}\label{decompw}
\begin{split}
-&\partial_x P_+\big[W(  P_- \partial_x u) \big]-\frac{i}{2}P_+\big[g e^{-\frac{i}{2}F}\big]-\frac{1}{4}P_+\big[\big(-\frac{1}{2}P_0 (u^2)+G\big) ue^{-\frac{i}{2}F}\big]\\
&=\NLI+\NLII+\NLIII.
\end{split}
\end{equation}

In our arguments we use the integral equivalent form of the solution of  \eqref{equaw}, that is,
\begin{equation}\label{equawint}
\begin{split}
w(t)&=\W(t)w_0+\int_0^t \W(t-t') (\NLI +\NLII +\NLIII)(t')\,dt',\\
&=\W(t)w_0+\IW(t),
\end{split}
\end{equation}
where $\W(t)$ denotes the unitary group associated to the linear Schr\"odinger equation.

\begin{remark}
We observe that the estimates in Lemmas \ref{lema1}-\ref{strichartz} also hold for solution of the IVP \eqref{equaw}.
\end{remark}

We will end this section giving the statement of the key (main) bilinear estimate we will use in our analysis.  It was proved by
Molinet and Pilod (see Proposition 3.5 of \cite{MolinetPilodBO}).
\begin{lemme}
\label{bilinestim}
Let  $\theta \in (0,1/8)$.
We have uniformly for $0\leq T\leq 1$
\bna
\|\partial_x P_+\big[W(  P_- \partial_x u) \big]\|_{X_2'}\leq CT^{\theta} \|(u,w)\|_{X}^2.
\ena
\end{lemme}

We shall remark that this estimate is a slight modification of the estimates in \cite{MolinetPilodBO} since we also have a factor $T^{\theta}$ that will be very useful to make some bootstrap and absorption. In the appendix, we sketch the modification of proof of Molinet-Pilod \cite{MolinetPilodBO} that allows us to obtain that term.

In what follows, we will fix $\theta \in (0,1/8)$ such that Lemma \ref{bilinestim} and (\ref{estimaYL2}) hold for $0< T\leq 1$.

\section{General estimates}
In this section we will derive  estimates needed in our argument which are independent of the size of the initial data.

\subsection{Estimates on $w$}

\begin{lemme}\label{lemma0} Let $w$ be the solution of \eqref{equaw}. Then it holds that
\begin{equation*}
\|w\|_{L^{\infty}([0,T],L^2)}\leq\|w\|_{Z_T^{0,0}}\leq \|(u,w)\|_{X},
\end{equation*}
and 
\begin{equation*}
\|w\|_{L^{4}([0,T]\times \T)}\leq \|w\|_{X_T^{0,1/2}}\leq \|(u,w)\|_{X}.
\end{equation*}
\end{lemme}

\begin{proof} It follows readily from the space definition and Lemma \ref{strichartz}.
\end{proof}

\begin{lemme}\label{estimatesw}
For $w$ in \eqref{equaw}  and $\theta\in (0,1/8)$ the following estimates hold
\begin{equation}\label{estimw2}
\|w\|_{X_2}\le  \|w_0\|_{L^2}+CT^{\theta}\big( \|g\|_{L^{2}([0,T],L^2)}+\|g\|_{L^{2}([0,T],L^2)}^2+\|(u,w)\|_{X}^2+\|(u,w)\|_{X}^3\big).
\end{equation}
Moreover,
\begin{equation}\label{estimwlinfin}
\begin{split}
\|w-w_{L}\|_{X_2}\leq CT^{\theta}\big(\|g\|_{L^{2}([0,T],L^2)}\|(u,w)\|_{X}+\|(u,w)\|_{X}^2+\|(u,w)\|_{X}^3\big)
\end{split}
\end{equation}
where $w_L$
is solution of
\begin{equation}\label{linearwL}
\begin{cases}
\partial_t w_L-i\partial_x^2w_L=-\frac{i}{2}P_+g\\
w_L(0)=w_0.
\end{cases}
\end{equation}
\end{lemme}

\label{subsectwestimgauge}
\begin{proof}
We first prove that
\begin{equation}\label{estimatesw1}
\|\IW\|_{X_2}\le C\,T^{\theta}\Big(\|(u,w)\|_{X}^2+\|g\|_{L^2([0,T]\times \T)}+
\|(u,w)\|_{X}^3+\|g\|_{L^{2}([0,T],L^2)}\|(u,w)\|_{X}\Big).
\end{equation}

Using Lemma \ref{lemma1b} we obtain
\begin{equation}\label{group1}
\|\IW\|_{X_2}\lesssim \|\NLI\|_{X_2'}+\|\NLII\|_{X_2'}+\|\NLIII\|_{X_2'}.
\end{equation}

From Lemma \ref{bilinestim} we have that
\bna
\|\NLI\|_{X_2'}\leq CT^{\theta}\|(u,w)\|_{X}^2.
\ena

The term $\NLII$ will be crucial in our analysis. We rewrite it as
\begin{equation}\label{crucial}
\begin{split}
\NLII&=-\frac{i}{2}P_+\big[g e^{-\frac{i}{2}F}\big]=-\frac{i}{2}P_+g-\frac{i}{2}P_+\big[g \big(e^{-\frac{i}{2}F}-1\big)\big]\\
&=\NLII_L+\NLII_{NL}
\end{split}
\end{equation}

Hence Lemma \ref{lemma-estimaYL2} yields
\begin{equation*}
\begin{split}
\|\NLII\|_{X_2'}&\leq CT^{\theta}\|\NLII_L\|_{L^2([0,T]\times \T)}+\|\NLII_{NL}\|_{L^2([0,T]\times \T)}\\
&\leq CT^{\theta}\big(\|g\|_{L^2([0,T]\times \T)} + \|g\|_{L^2([0,T]\times \T)}\|e^{-\frac{i}{2}F}-1\|_{L^{\infty}([0,T]\times \T)})\\
&\leq CT^{\theta}\big(\|g\|_{L^2([0,T]\times \T)}+\|g\|_{L^2([0,T]\times \T)}\|u\|_{L^{\infty}([0,T],L^2)}\big)\\
&\leq  CT^{\theta}\big(\|g\|_{L^2([0,T]\times \T)}+\|g\|_{L^2([0,T]\times \T)}\|(u,w)\|_{X}\big).
\end{split}
\end{equation*}

For $\NLIII$, we also use Lemma \ref{lemma-estimaYL2}
\begin{equation*}
\begin{split}
\|\NLIII\|_{X_2'}&\leq C\,T^{\theta}\|\NLIII\|_{L^2([0,T]\times \T)}\leq CT^{\theta}\|-\frac{1}{2}P_0 (u^2)+G\|_{L^{2}([0,T],L^{\infty})}\|u\|_{L^{\infty}([0,T],L^2)}\\
&\leq C\,T^{\theta}\big(\|u\|_{L^{\infty}([0,T],L^2)}^2+\|g\|_{L^{2}([0,T],L^2)}\big)\|u\|_{L^{\infty}([0,T],L^2)}\\
&\leq  C\,T^{\theta}\big(\|(u,w)\|_{X}^3+\|g\|_{L^{2}([0,T],L^2)}\|(u,w)\|_{X}\big).
\end{split}
\end{equation*}

Combining the above estimates with \eqref{group1} yields  \eqref{estimatesw1}.

Finally, using Duhamel formulation, Lemmas \ref{lema1}, \ref{lemma1b}, \eqref{group1}  and standard estimates we have
\begin{equation}\label{estimw2b}
\begin{split}
\|&w\|_{X_2}=\|w\|_{X^{0,1/2}}+\|w\|_{Z^{0,0}}\\
&\leq \|w_0\|_{L^2}+CT^{\theta}\big( \|g\|_{L^{2}([0,T],L^2)}+\|g\|_{L^{2}([0,T],L^2)}\|(u,w)\|_{X}+\|(u,w)\|_{X}^2+\|(u,w)\|_{X}^3\big)\\
&\leq \|w_0\|_{L^2}+CT^{\theta}\big( \|g\|_{L^{2}([0,T],L^2)}+\|g\|_{L^{2}([0,T],L^2)}^2+\|(u,w)\|_{X}^2+\|(u,w)\|_{X}^3\big).
\end{split}
\end{equation}

Similarly,
\begin{equation}\label{estimwliplin}
\begin{split}
\|w-w_{L}\|_{X_2}&=\|w-w_{L}\|_{X^{0,1/2}}+\|w-w_{L}\|_{Z^{0,0}}\\
&\leq CT^{\theta}\big(\|g\|_{L^{2}([0,T],L^2)}\|(u,w)\|_{X}+\|(u,w)\|_{X}^2+\|(u,w)\|_{X}^3\big)
\end{split}
\end{equation}
where $w_{L}$ is as in \eqref{linearwL}.

\end{proof}

Next we derive the Lipschitz estimates corresponding to the $X_2$ norm.

\begin{lemme}[Lipschitz estimates]\label{estimatesw2}
Let $w_1$ and $w_2$ be solutions of \eqref{equaw} and $\theta\in (0,1/8)$. Then the following estimates hold:
\begin{equation}\label{estimwlip}
\begin{split}
\|w_1-w_2\|_{X_2}&\leq \|w_{0,1}-w_{0,2}\|+CT^{\theta}\,\|g_1-g_2\|_{L^2([0,T]\times \T)}(1+\|U_1\|_X)\\
&\;\;\;\;+CT^{\theta}\|U_1-U_2\|_{X} \Big(\|U_1\|_{X}+\|U_2\|_{X}+(\|U_1\|_{X}+\|U_2\|_X)^2\\
&\;\;\;\;+\|U_2\|_{X}^3+\|g_2\|_{L^2([0,T]\times \T)}(1+\|U_2\|_X)\Big)
\end{split}
\end{equation}
and
\begin{equation}\label{estimwlinfinb}
\begin{split}
\|(w_1-w_{L,1})-&(w_2-w_{L,2})\|_{X_2}\le CT^{\theta}\,\|g_1-g_2\|_{L^2([0,T]\times\T)}\|U_1\|_{X}\\
& +CT^{\theta}\|U_1-U_2\|_{X} \Big(\|U_1\|_{X}+\|U_2\|_{X}+(\|U_1\|_{X}+\|U_2\|_X)^2\\
&+\|U_2\|_{X}^3+\|g_2\|_{L^2([0,T]\times \T)}(1+\|U_2\|_X)\Big)
\end{split}
\end{equation}
where $w_{L,1}$ and $w_{L,2}$ are  solutions of \eqref{linearwL}.
\end{lemme}

\begin{proof}  We first show that
\begin{equation}\label{lipschw1}
\begin{split}
\|\IW_1-\IW_2\|_{X_2}&\le CT^{\theta}\,\|g_1-g_2\|_{L^2([0,T]\times \T)}(1+\|U_1\|_X)\\
&\;\;\;\;+CT^{\theta}\|U_1-U_2\|_{X} \Big(\|U_1\|_{X}+\|U_2\|_{X}+(\|U_1\|_{X}+\|U_2\|_X)^2\\
&\;\;\;\;+\|U_2\|_{X}^3+\|g_2\|_{L^2([0,T]\times \T)}(1+\|U_2\|_X)\Big).\\
\end{split}
\end{equation}
Let $w_1$ and $w_2$ be two solutions of \eqref{equaw}. We use the notation \eqref{decompw}. From  Lemma \ref{lemma1b} we obtain
\begin{equation}\label{group1b}
\|\IW_1-\IW_2\|_{X_2}\lesssim \|\NLI\|_{X_2'}+\|\NLII\|_{X_2'}+\|\NLIII\|_{X_2'}.
\end{equation}

Since $\NLI$ is bilinear,  we get the estimates by writing
\begin{equation*}
\begin{split}
\NLI_1-\NLI_2&=-\partial_x P_+\big[W_1(  P_- \partial_x u_1) \big]+\partial_x P_+\big[W_2(  P_- \partial_x u_2) \big]\\
&=-\partial_x P_+\big[(W_1-W_2)(  P_- \partial_x u_1) \big]-\partial_x P_+\big[W_2(  P_- \partial_x (u_1-u_2)) \big].
\end{split}
\end{equation*}
Thus
\begin{equation*}
\begin{split}
\|\NLI_1-\NLI_2\|_{X_2'}&\leq CT^{\theta}\|w_1-w_2\|_{X_2}\|u_1\|_{X_1}+\|w_2\|_{X_2}\|u_1-u_2\|_{X_1}\\
&\leq CT^{\theta}\|U_1-U_2\|_{X}\big(\|U_1\|_{X}+\|U_2\|_{X}\big).
\end{split}
\end{equation*}

We write $\NLII=\NLII_L+ \NLII_{NL}$ as in \eqref{crucial} . Since $\NLII_L$ is linear, we obtain
\begin{equation*}
\|\NLII_{L,1}-\NLII_{L,2}\|_{X_2'}\leq CT^{\theta}\|\NLII_{L,1}-\NLII_{L,2}\|_{L^2([0,T]\times \T)}\leq CT^{\theta}\|g_1-g_2\|_{L^2([0,T]\times \T)}.
\end{equation*}
For $II_{NL}$, we write
\begin{equation*}
\begin{split}
\NLII_{NL,1}-\NLII_{NL,2}&=-\frac{i}{2}P_+\big[\big[g_1 \big(e^{-\frac{i}{2}F_1}-1\big)\big]-\big[g_2 \big(e^{-\frac{i}{2}F_2}-1\big)\big]\big]\\
&=-\frac{i}{2}P_+\big[\big[(g_1-g_2) \big(e^{-\frac{i}{2}F_1}-1\big)\big]+\big[g_2 \big(e^{-\frac{i}{2}F_1}-e^{-\frac{i}{2}F_2}\big)\big]\big]
\end{split}
\end{equation*}
\begin{equation*}
\begin{split}
\|\NLII_{NL,1}-&\NLII_{NL,2}\|_{X_2'}\leq  CT^{\theta}\|\NLII_{NL,1}-\NLII_{NL,2}\|_{L^2([0,T]\times \T)}\\
&\leq CT^{\theta}\big(\|g_1-g_2\|_{L^2([0,T]\times \T)}\|u_1\|_{L^{\infty}([0,T],L^2)}+\|g_2\|_{L^2([0,T]\times \T)}\|u_1-u_2\|_{L^{\infty}([0,T],L^2)}\big)\\
&\leq CT^{\theta}\big(\|g_1-g_2\|_{L^2([0,T]\times \T)}\|U_1\|_{X}+\|g_2\|_{L^2([0,T]\times \T)}\|U_1-U_2\|_{X}\big).
\end{split}
\end{equation*}
For $\NLIII$, we write
\begin{equation*}
\begin{split}
\NLIII_1-\NLIII_2=&P_+\Big(\big[\big(-\frac{1}{2}P_0 (u_1^2)+G_1\big) u_1e^{-\frac{i}{2}F_1}\big]-\big[\big(-\frac{1}{2}P_0 (u_2^2)+G_2\big) u_2e^{-\frac{i}{2}F_2}\big]\Big)\\
=&P_{+}\big[\big(-\frac{1}{2}P_0 ((u_1-u_2)(u_1+u_2))+G_1-G_2\big) u_1e^{-\frac{i}{2}F_1}\big]\\
&+P_{+}\big[\big(-\frac{1}{2}P_0 (u_2^2)+G_2\big) \big((u_1-u_2)e^{-\frac{i}{2}F_1}+u_2(e^{-\frac{i}{2}F_1}-e^{-\frac{i}{2}F_2})\big)\big]
\end{split}
\end{equation*}
Using the same estimates (the only difference is that we use \refp{estimlipsexp} for the Lipschitz estimate of the term $e^{-\frac{i}{2}F})$, we have then
\begin{equation*}
\begin{split}
&\|\NLIII_1-\NLIII_2\|_{X_2'}\leq CT^{\theta}\|\NLIII_1-\NLIII_2\|_{L^2([0,T]\times \T)}\\
&\leq CT^{\theta}\|-\frac{1}{2}P_0 ((u_1-u_2)(u_1+u_2))+G_1-G_2\|_{L^{2}([0,T],L^{\infty})}\|u_1\|_{L^{\infty}([0,T],L^2)}\\
&\;\;\;+CT^{\theta}\|-\frac{1}{2}P_0 (u_2^2)+G_2\|_{L^{2}([0,T],L^{\infty})}\Big(\|u_1-u_2\|_{L^{\infty}([0,T],L^2)}\\
&\;\;\;+\|u_2\|_{L^{\infty}([0,T],L^2)}\|e^{-\frac{i}{2}F_1}-e^{-\frac{i}{2}F_2}\|_{L^{\infty}([0,T]\times \T)}\Big)
\end{split}
\end{equation*}
\begin{equation*}
\begin{split}
&\leq CT^{\theta}\|u_1-u_2\|_{L^4([0,T]\times\T)}(\|u_1\|_{L^4([0,T]\times\T)}+\|u_2\|_{L^4([0,T]\times\T)}\big)\|u_1\|_{L^{\infty}([0,T],L^2)}\\
&\;\;\;+CT^{\theta}\|g_1-g_2\|_{L^{2}([0,T],L^2)}\|u_1\|_{L^{\infty}([0,T],L^2)}\\
&\;\;\;+CT^{\theta}\big(\|u_2\|_{L^{\infty}([0,T],L^2)}^2+\|g_2\|_{L^{2}([0,T],L^2)}\big)(1+\|u_2\|_{L^{\infty}([0,T],L^2)})\|u_1-u_2\|_{L^{\infty}([0,T],L^2)}\\
&\leq CT^{\theta}\Big( \|U_1-U_2\|_{X}\big(\|U_1\|_{X}+\|U_2\|_{X}\big)^2+\|g_1-g_2\|_{L^{2}([0,T],L^2)}\|U_2\|_{X}\Big)\\
&\;\;\;+CT^{\theta}\,\Big(\|U_2\|_{X}^3+\|g_2\|_{L^{2}([0,T],L^2)}(1+\|U_2\|_{X})\Big)\|U_1-U_2\|_{X}.\\
\end{split}
\end{equation*}

The estimates \eqref{estimwlip} and \eqref{estimwlinfinb} follow by  using Duhamel formulation, Lemmas \ref{lema1}, \ref{lemma1b}, \eqref{group1}  and standard estimates.
\end{proof}

\subsection{The estimates coming from the original Benjamin-Ono equation} 

We consider again the IVP
\begin{equation}\label{soluBO}
\begin{cases}
\partial_t u+\H \partial_x^2 u=\frac{1}{2}\partial_x (u^2)+g\\
u(0)=u_0.
\end{cases}
\end{equation}

\begin{lemme}\label{estimsolBO}
Let $u$ be a solution of \eqref{soluBO}. Then the following estimate holds:
\begin{equation}\label{soluBO1}
\|u\|_{X^{-1,1}_T}\le C\|u_0\|_{L^2}+\|(u,w)\|_{X}^2+\|g\|_{L^2([0,T],L^2)}.
\end{equation}

Moreover, if we denote by $u_L$ the solution of
\begin{equation*}
\begin{cases}
\partial_t u_L+\H \partial_x^2 u_L=g\\
u(0)=u_0,
\end{cases}
\end{equation*}
we have
\begin{equation}\label{soluBO1b}
\|u-u_L\|_{X^{-1,1}_T}\leq \| (u,w)\|_{X}^2.
\end{equation}
\end{lemme}

\begin{proof} Using the definition of the space $X_T^{-1,1}$ we have the estimate
\bna
\|u\|_{X^{-1,1}_T}\le \|u\|_{L^2([0,T],H^{-1})}+\|\frac{1}{2}\partial_x(u^2)+g\|_{L^2([0,T],H^{-1})}
\ena

Using energy estimates we get 
\bna \|u\|_{L^2([0,T],H^{-1})}\le 
C\|u_0\|_{L^2}+C\|\frac{1}{2}\partial_x
(u^2)+g\|_{L^1([0,T],H^{-1})}. 
\ena 

Finally, we deduce that
\begin{equation}\label{soluBO2}
\begin{split}
\|u\|_{X^{-1,1}_T}&\leq C\|u_0\|_{L^2}+\|\frac{1}{2}\partial_x (u^2)+g\|_{L^2([0,T],H^{-1})}\\
&\leq C\|u_0\|_{L^2}+\| u^2\|_{L^2([0,T],L^2)}+\|g\|_{L^2([0,T],L^2)}\\
&\leq C\|u_0\|_{L^2}+\| u\|_{L^4([0,T],L^4)}^2+\|g\|_{L^2([0,T],L^2)}\\
&\leq C\|u_0\|_{L^2}+\| (u,w)\|_{X}^2+\|g\|_{L^2([0,T],L^2)}.
\end{split}
\end{equation}

\end{proof}

Using the previous lemma, it is not difficult to deduce the Lipschitz estimates in the $X^{-1,1}_T$-norm. Indeed,

\begin{lemme}[Lipschitz estimates]\label{lipschequ} The next estimates hold
\begin{equation}\label{estimulip}
\begin{split}
\|u_1-u_2\|_{X^{-1,1}_T}&\lesssim \|u_{0,1}-u_{0,2}\|_{L^2}+ \|g_1-g_2\|_{L^{2}([0,T]\times \T)}\\
&\;\;\;+\| u_1-u_2\|_{L^4([0,T]\times \T)}\big(\|(u_1,w_1)\|_{X}+\|(u_2,w_2)\|_{X}\big)
\end{split}
\end{equation}
and
\begin{equation}\label{estimuliplin}
\|u_1-u_{L,1}-(u_2-u_{L,2})\|_{X^{-1,1}_T}\lesssim \| (u_1,w_1)-(u_2,w_2)\|_{X}\big(\| (u_1,w_1)\|_{X}+\|(u_2,w_2)\|_{X}\big).
\end{equation}
\end{lemme}
\begin{proof}
The arguments in Lemma  \ref{estimsolBO} with the required modifications yield the inequalities.
\end{proof}

\subsection{Estimates coming from the gauge transform} 

Since  $w=-\frac{i}{2}P_+(ue^{-\frac{i}{2}F})$ we have that $ue^{-\frac{i}{2}F}=2iw+P_{\leq 0}(ue^{-\frac{i}{2}F})$.
Hence
\begin{equation}\label{inversionuw1}
u=2iwe^{\frac{i}{2}F}+e^{\frac{i}{2}F}P_{\leq 0}(ue^{-\frac{i}{2}F}).
\end{equation}
But the second term is bad, it roughly says (at the first order in $u$) that for the negative frequencies, $u=u$, and so it does not allow to invert. But for positive frequencies
\begin{equation}\label{ABgauge}
\begin{split}
P_+u&=2iP_+\big[we^{\frac{i}{2}F}\big]+P_+\big[P_+(e^{\frac{i}{2}F})P_{\leq 0}(ue^{-\frac{i}{2}F})\big]\\
&=2iP_+w+2iP_+\big[w(e^{\frac{i}{2}F}-1)\big]+P_+\big[P_+(e^{\frac{i}{2}F})P_{\leq 0}(ue^{-\frac{i}{2}F})\big]\\
&=2iw+A+B,
\end{split}
\end{equation}
 where we have used $P_+\big[P_{\leq 0}(e^{\frac{i}{2}F})P_{\leq 0}(ue^{-\frac{i}{2}F})\big]=0 $.
In this case, the third term is quadratic in $u$. Indeed,
\begin{equation}\label{projposB}
\begin{split}
B=P_+\big[P_+(e^{\frac{i}{2}F})P_{\leq 0}(ue^{-\frac{i}{2}F})\big]&=P_+\big[P_+(e^{\frac{i}{2}F}-1)P_{\leq 0}(ue^{-\frac{i}{2}F})\big]+P_+\big[P_{\leq0}(ue^{-\frac{i}{2}F})\big]\\
&=P_+\big[P_+(e^{\frac{i}{2}F}-1)P_{\leq0}(ue^{-\frac{i}{2}F})\big].
\end{split}
\end{equation}

\begin{lemme}\label{projection} It holds that
\begin{equation}\label{estimlinPua}
\|P_+u\|_{L^{\infty}([0,T],L^2)}+\|P_+u\|_{L^{4}([0,T]\times \T)}\leq \|w\|_{X_2}+\|(u,w)\|_{X}^2
\end{equation}
and
\begin{equation}\label{estimlinPub}
\|P_+u-2iw\|_{L^{\infty}([0,T],L^2)}+\|P_+u-2iw\|_{L^{4}([0,T]\times \T)}\leq \|(u,w)\|_{X}^2.
\end{equation}
\end{lemme}

\begin{proof} We use the decomposition \eqref{ABgauge} to estimate $P_{+}$. From Lemma \ref{lemma0} it follows that
\begin{equation}
\|2iw\|_{L^{\infty}([0,T],L^2)}+\|2iw\|_{L^{4}([0,T]\times \T)}\leq \|w\|_{Z^{0,0}_T}+\|w\|_{X^{0,1/2}_T}\leq \|w\|_{X_2}.
\end{equation}

Since $P_+$ is a pseudo-differential operator of order $0$, it maps $L^4$ into itself. Hence
\begin{equation}
\begin{split}
\|A\|_{L^{\infty}([0,T],L^2)}+&\|A\|_{L^{4}([0,T]\times \T)}\\
&\lesssim \big(\|w\|_{L^{\infty}([0,T],L^2)}+\|w\|_{L^{4}([0,T]\times \T)}\big)\|e^{\frac{i}{2}F}-1\|_{L^{\infty}([0,T]\times \T)}\\
&\lesssim \big(\|w\|_{Z^{0,0}_T}+\|w\|_{X^{0,1/2}_T}\big)\|u\|_{L^{\infty}([0,T],L^2)}\leq \|(u,w)\|_{X}^2
\end{split}
\end{equation}
and
\begin{equation*}
\begin{split}
\|B\|_{L^{\infty}([0,T],L^2)}+\|B\|_{L^{4}([0,T]\times \T)}&\lesssim \|P_{+}\big(e^{\frac{i}{2}F}-1\big)\|_{L^{\infty}([0,T]\times \T)}\big(\|u\|_{L^{\infty}([0,T],L^2)}
+\|u\|_{L^{4}([0,T]\times \T)}\big)\\
&\lesssim \|(u,w)\|_{X}^2.
\end{split}
\end{equation*}

Combining these estimates we obtain the desired inequalities.
\end{proof}

Next we obtain  Lipschitz  estimates for the terms involving the operator  $P_+$. More precisely,
\begin{lemme}[Lipschitz Estimates] It holds that
\begin{equation}\label{estimlinPuLip}
\begin{split}
\|P_+u_1-&2iw_1-(P_+u_2-2iw_2) \|_{L^{\infty}([0,T],L^2)}+\|P_+u_1-2iw_1-(P_+u_1-2iw_2) \|_{L^{4}([0,T]\times \T)}\\
&\leq \|(u_1,w_1)-(u_2,w_2) \|_{X}\big(\|(u_1,w_1) \|_{X}+\|(u_2,w_2) \|_{X}+\|(u_2,w_2) \|_{X}^2\big)
\end{split}
\end{equation}
and
\begin{equation}\label{estimPuLip}
\begin{split}
&\|P_+u_1-P_+u_2 \|_{L^{\infty}([0,T],L^2)}+\|P_+u_1-P_+u_2 \|_{L^{4}([0,T]\times \T)}\\
&\leq \|w_1-w_2 \|_{X_2}+\|(u_1,w_1)-(u_2,w_2) \|_{X}\big(\|(u_1,w_1) \|_{X}+\|(u_2,w_2) \|_{X}+\|(u_2,w_2) \|_{X}^2\big).
\end{split}
\end{equation}
\end{lemme}

\begin{proof} Recall the decomposition \eqref{ABgauge}. To establish the corresponding Lipschitz estimates we write
\begin{equation*}
\begin{split}
A_1-A_2&=2iP_+\big[w_1(e^{\frac{i}{2}F_1}-1)\big]-2iP_+\big[w_2(e^{\frac{i}{2}F_2}-1)\big]\\
&=2iP_+\big[(w_1-w_2)(e^{\frac{i}{2}F_1}-1)+w_2\big(e^{\frac{i}{2}F_1}-e^{\frac{i}{2}F_2}\big)\big].
\end{split}
\end{equation*}
Hence
\begin{equation*}
\begin{split}
\|A_1&-A_2\|_{L^{\infty}([0,T],L^2)}+\|A_1-A_2\|_{L^{4}([0,T]\times \T)}\\
\leq & \big(\|w_1-w_2\|_{L^{\infty}([0,T],L^2)}+\|w_1-w_2\|_{L^{4}([0,T]\times \T)}\big)\|e^{\frac{i}{2}F_1}-1\|_{L^{\infty}([0,T]\times \T)}\\
&+\big(\|w_2\|_{L^{\infty}([0,T],L^2)}+\|w_2\|_{L^{4}([0,T]\times \T)}\big)\|e^{\frac{i}{2}F_1}-e^{\frac{i}{2}F_2}\|_{L^{\infty}([0,T]\times \T)}\\
\leq & \big(\|w_1-w_2\|_{Z^{0,0}_T}+\|w_1-w_2\|_{X^{0,1/2}_T}\big)\|u_1\|_{L^{\infty}([0,T],L^2)}\\
&+ \big(\|w_2 \|_{Z^{0,0}_T}+\|w_2\|_{X^{0,1/2}_T}\big)\|u_1-u_2\|_{L^{\infty}([0,T],L^2)}\\
\leq &\|U_1-U_2\|_{X}\big(\|U_1\|_{X}+\|U_2\|_{X}\big).
\end{split}
\end{equation*}

On the other hand, we write
\begin{equation*}
\begin{split}
&B_1-B_2=P_+\big[P_+(e^{\frac{i}{2}F_1}-1)P_{\leq0}(u_1e^{-\frac{i}{2}F_1})-P_+(e^{\frac{i}{2}F_2}-1)P_{\leq0}(u_2e^{-\frac{i}{2}F_2})\big]\\
&=P_+\big[P_+(e^{\frac{i}{2}F_1}-e^{\frac{i}{2}F_2})P_{\leq0}(u_1e^{-\frac{i}{2}F_1})+P_+(e^{\frac{i}{2}F_2}-1)P_{\leq 0}
\big((u_1-u_2)e^{-\frac{i}{2}F_1}+u_2(e^{-\frac{i}{2}F_1}-e^{-\frac{i}{2}F_2})\big)\big].
\end{split}
\end{equation*}
Thus
\begin{equation*}
\begin{split}
\|B_1&-B_2\|_{L^{\infty}([0,T],L^2)}+\|B_1-B_2\|_{L^{4}([0,T]\times \T)}\\
&\leq C\,\Big(\|U_1\|_X+\|U_2\|_X+(\|U_1\|_{X}+\|U_2\|_{X})^2\Big)\|U_1-U_2\|_{X}.\\
\end{split}
\end{equation*}

These estimates lead to inequalities \eqref{estimlinPuLip} and \eqref{estimPuLip}.
\end{proof}

\vskip0.3cm
\section{Small data estimates}
The control results as well as the proof of Theorem \ref{sourcelwp} depend on the size data. 
In this section we aim to establish all the estimates related to small data needed in our
arguments. We will also give the proof of Theorem \ref{sourcelwp} corresponding
to small data.

Since we are just considering real valued functions we have $P_-u=\overline{P_+u}$, thus
\bna
 \|u \|_{L^{\infty}([0,T],L^2)}+\|u \|_{L^{4}([0,T]\times \T)} \approx \|P_+u \|_{L^{\infty}([0,T],L^2)}+\|P_+u \|_{L^{4}([0,T]\times \T)}.
\ena
Gathering the information from Lemmas \ref{estimatesw}, \ref{projection} and \ref{estimsolBO}  yields
\begin{equation*}
\begin{split}
\|w&\|_{X_2}\leq \|w_0\|_{L^2}+C\big(\|g\|_{L^{2}([0,T],L^2)}+\|g\|_{L^{2}([0,T],L^2)}^2+\|(u,w)\|_{X}^2+\|(u,w)\|_{X}^3\big)\\
\|u&\|_{L^{\infty}([0,T],L^2)}+\|u\|_{L^{4}([0,T]\times \T)}\leq\|w\|_{X_2}+\|(u,w)\|_{X}^2\\
\|u&\|_{X^{-1,1}_T}\leq \|u_0\|_{L^2}+\| (u,w)\|_{X}^2+\|g\|_{L^2([0,T],L^2)}.
\end{split}
\end{equation*}
and finally
\bna
\|(u,w) \|_{X}\lesssim 
 \|u_0\|_{L^2}+ \|g \|_{L^{2}([0,T],L^2)}+\|g \|_{L^{2}([0,T],L^2)}^2+\|(u,w) \|_{X}^2+\|(u,w) \|_{X}^3
\ena
If $\|g \|_{L^{2}([0,T],L^2)}$ is small enough, it follows that
\bna
\|(u,w) \|_{X}\lesssim  \|u_0\|_{L^2}+ \|g \|_{L^{2}([0,T],L^2)}+\|(u,w) \|_{X}^2+\|(u,w) \|_{X}^3.
\ena
Then, since all these quantities are continuous in $T$ and have value at $t=0$ bounded by $C\|u_0\|_{L^2}$, see Lemma \ref{continuiteXt}, we can apply a bootstrap argument to 
get for $\|u_0\|$ and $\|g \|_{L^{2}([0,T],L^2)}$,  small enough, that
\begin{equation}\label{lipschbound}
\|(u,w) \|_{X}\lesssim \|u_0\|_{L^2}+ \|g \|_{L^{2}([0,T],L^2)}.
\end{equation}

\subsection{First order estimates}  Using the Lipschitz bound in the previous section we will establish some estimates useful to prove the control results.

The estimates \eqref{estimlinPub} and \eqref{estimwlinfinb} yield
\begin{equation*}
\begin{split}
\|P_+u-2iw \|_{L^{\infty}([0,T],L^2)}&\leq \|(u,w) \|_{X}^2 \leq \|u_0\|_{L^2}^2+\|g \|_{L^{2}([0,T],L^2)}^2\\
\|w-w_{L} \|_{X_2}=\|w-w_{L} \|_{X^{0,1/2}}+\|w-w_{L} \|_{Z^{0,0}}&\leq \|g \|_{L^{2}([0,T],L^2)}\|(u,w) \|_{X}+\|(u,w) \|_{X}^2.
\end{split}
\end{equation*}
Finally, we have by triangular inequality (noting that $Z^{0,0}\subset L^{\infty}([0,T],L^2)$)
\bna
\|P_+u-2iw_L \|_{L^{\infty}([0,T],L^2)}&\leq & \|u_0\|_{L^2}^2+\|g \|_{L^{2}([0,T],L^2)}^2
\ena
where
\begin{equation*}
\begin{cases}
\partial_t w_L-i\partial_x^2w_L=\frac{i}{2}P_+g,\\
w_L(0)=w_0
\end{cases}
\end{equation*}

In particular, since $u$ is real valued, $P_-u=\overline{P_+u}$, so
\bna
\|P_-u-\overline{2iw_L} \|_{L^{\infty}([0,T],L^2)}&\leq &  \|u_0\|_{L^2}^2+\|g \|_{L^{2}([0,T],L^2)}^2.
\ena
But, we notice that $u=P_+u+P_-u$ and $u_L=2iw_L+\overline{2iw_L}$. In particular, it follows that
\begin{equation}\label{estimlinaru}
\|u-u_L \|_{L^{\infty}([0,T],L^2)}\lesssim \|u_0\|_{L^2}^2+\|g \|_{L^{2}([0,T],L^2)}^2.
\end{equation}
\vskip10pt

\subsection{Lipschitz estimates}
If $\|u_{0,i}\|$, $\|g_i\|_{L^{2}([0,T],L^2)}$, $i=1,2$, are small enough, we still have
\bna
\|(u_i,w_i) \|_{X}\lesssim \|u_{0,i}\|+ \|g_i \|_{L^{2}([0,T],L^2)}.
\ena

Assume that $\|u_{0,i}\|+\|g_i\|_{L^{2}([0,T],L^2)}\leq \e$, with $\e$ small. Then
\refp{estimwlip} becomes
\begin{equation*}
\|w_1-w_2\|_{X_2}\lesssim \|u_{0,1}-u_{0,2}\|+ \|g_1-g_2\|_{L^2([0,T]\times \T)}+\e\|(u_1,w_1)-(u_2,w_2)\|_{X}.
\end{equation*}
Inequality \refp{estimulip} becomes
\begin{equation*}
\|u_1-u_2\|_{X^{-1,1}_T}\lesssim  \|u_{0,1}-u_{0,2}\|+\|g_1-g_2\|_{L^{2}([0,T]\times \T)}+\e\|(u_1,w_1)-(u_2,w_2)\|_{X}
\end{equation*}
and inequality \refp{estimPuLip} becomes
\begin{equation*}
\begin{split}
\|P_+u_1&-P_+u_2\|_{L^{\infty}([0,T],L^2)}+\|P_+u_1-P_+u_1\|_{L^{4}([0,T]\times \T)}\\
&\lesssim \|w_1-w_2\|_{X_2}+\e\|(u_1,w_1)-(u_2,w_2)\|_{X}.
\end{split}
\end{equation*}

Thus gathering all these estimates of the $X$--norm components  (using again that $u$ is real valued and that the estimates about
$P_+u$ are sufficient  for the $L^{\infty}([0,T],L^2)$ and $L^{4}([0,T]\times \T)$ norms), we get
\bna
\| (u_1,w_1)-(u_2,w_2)\|_{X}\lesssim  \|u_{0,1}-u_{0,2}\|+ \|g_1-g_2\|_{L^{2}([0,T]\times \T)}+\e\|(u_1,w_1)-(u_2,w_2)\|_{X}
\ena
which gives by absorption for $\e$ small enough
\begin{equation}\label{estimlipuw}
\| (u_1,w_1)-(u_2,w_2)\|_{X}\lesssim  \|u_{0,1}-u_{0,2}\|+ \|g_1-g_2\|_{L^{2}([0,T]\times \T)}.
\end{equation}

For the estimates of the second order terms, we have from \refp{estimwlinfinb} and \refp{estimlipuw}
\bna
&&\|w_1-w_{L,1}-(w_2-w_{L,2}) \|_{X_2}\lesssim \e\left(\|u_{0,1}-u_{0,2}\|+\|g_1-g_2 \|_{L^2([0,T]\times \T)}\right)
\ena
and from \refp{estimlinPuLip}  that
\bna
\|P_+u_1-2iw_1-(P_+u_2-2iw_2) \|_{L^{\infty}([0,T],L^2)}
&\lesssim &\e\left(\|u_{0,1}-u_{0,2}\|+\|g_1-g_2 \|_{L^2([0,T]\times \T)}\right).
\ena

The triangular inequality yields
\bna
\|P_+u_1-2iw_{L,1}-(P_+u_2-2iw_{L,2}) \|_{L^{\infty}([0,T],L^2)} &\lesssim &\e\left(\|u_{0,1}-u_{0,2}\|+\|g_1-g_2 \|_{L^2([0,T]\times \T)}\right).
\ena

Once again, we notice that since $u_i$, $g_i$ are real valued, this implies
\begin{equation}\label{estimLipfinkin}
\|u_1-u_{L,1}-(u_2-u_{L,2}) \|_{L^{\infty}([0,T],L^2)} \lesssim \e\left(\|u_{0,1}-u_{0,2}\|+\|g_1-g_2 \|_{L^2([0,T]\times \T)}\right).
\end{equation}

\vskip10pt

\begin{proof}[Proof of Theorem \ref{sourcelwp} \rm(small data case)]  The main ingredient is the Lipschitz inequality \eqref {estimlipuw}. Once this is established, the proof of the theorem  follows from standard arguments. For a detailed and careful proof of it,  see  \cite{MolinetPilodBO}.
\end{proof}

\section{Large data estimates}

In this section we will complete the proof of Theorem \ref{sourcelwp} and establish some estimates for large data
useful for further analysis.

\subsection{Low-frequency estimate}

In this subsection, we establish some estimates of the low frequency component of the solution for eventually large data. These estimates  are necessary 
because  the gauge transform is not invertible at low frequency for large data.

Let $N\in \N^*$ be fixed. We will prove some estimates that might  depend on $N$.
\begin{lemme}
For $u$ solution of the IVP \eqref{soluBO}, we have the following uniform estimate for $0\leq T\leq 1$
\begin{equation}\label{estimLipschilowA}
\begin{split}
\|P_N&u\|_{L^{\infty}([0,T],L^2)}+\|P_Nu\|_{L^{4}([0,T]\times \T)}\\
&\leq C(N) \|u_0\|_{L^2}+C(N) \|g\|_{L^2([0,T]\times \T)}+C(N)T^{1/2}\|u\|_{L^{4}([0,T]\times \T)}^2,
\end{split}
\end{equation}
and for two different solutions
\begin{equation}\label{estimLipschilow}
\begin{split}
\|P_N&(u_1-u_2)\|_{L^{\infty}([0,T],L^2)}+\|P_N(u_1-u_2)\|_{L^{4}([0,T]\times \T)}\\
&\leq C(N) \|u_{0,1}-u_{0,2}\|_{L^2}+C(N) \|g_1-g_2\|_{L^2([0,T]\times \T)}\\
&\;\;\;+C(N)T^{1/2}\big(\|u_1\|_{L^{4}([0,T]\times \T)}+\|u_2\|_{L^{4}([0,T]\times \T)}\big)\|u_1-u_2\|_{L^{4}([0,T]\times \T)}.
\end{split}
\end{equation}
\end{lemme}

\begin{proof}
We first link the $L^4$ estimate to the $L^{\infty}_tL^2_x$ one by using Sobolev estimates, that is, 
\bna
\|P_Nu \|_{L^{4}([0,T]\times \T)}\leq CT^{1/4}\|P_Nu \|_{L^{\infty}([0,T],L^4)} \leq C\|P_Nu \|_{L^{\infty}([0,T],H^1)}\leq C(N)\|P_Nu \|_{L^{\infty}([0,T],L^2)}.
\ena

To estimate the $L^{\infty}_tL^2_x$ norm we use the solution  $u_N$  of the equation satisfied by $u_N=P_Nu$, i.e.
\begin{equation*}
\begin{cases}
\partial_t u_N+\H \partial_x^2 u_N=\frac{1}{2}P_N\partial_x (u^2)+P_Ng,\\
u_N(0)=P_N u_0.
\end{cases}
\end{equation*}
From semi-group estimates we obtain
\begin{equation*}
\begin{split}
\|P_Nu\|_{L^{\infty}([0,T],L^2)}&\leq C\|u_0\|_{L^2}+C\|P_N\partial_x (u^2)\|_{L^{1}([0,T],L^2)}+C\|P_Ng\|_{L^{1}([0,T],L^2)}\\
&\leq C\|u_0\|_{L^2}+C(N)T^{1/2}\|P_N(u^2)\|_{L^{2}([0,T],L^2)}+C\|g\|_{L^{1}([0,T],L^2)}\\
&\leq C\|u_0\|_{L^2}+C\|g\|_{L^{2}([0,T],L^2)}+C(N)T^{1/2}\|u\|_{L^{4}([0,T]\times \T)}^2.
\end{split}
\end{equation*}

The same argument leads to \eqref{estimLipschilow}.
\end{proof}

\subsubsection{Global large data estimates}
We change a little the decomposition (\ref{inversionuw1}). The point is, roughly speaking, to make the gauge transform \lq\lq invertible" at high frequency 
(the frequency will depend on the size of the data). For large $F$ the second term of the gauge transform \eqref{inversionuw1} is bad, and the gauge 
transform cannot be inverted for large $F$. But in some sense, it can be inverted at  large positive frequencies. Note that another way of obtaining global
 estimates for large data used by Molinet \cite{molinetBOT} was to use the scaling argument to get to some small data (because the equation is subcritical). 
 But this approach requires to have some estimates uniform on the lengths of the interval. Many of the estimates we used are indeed uniform, but this is not the 
 case of the Sobolev embedding $H^1\hookrightarrow L^{\infty}$. One {\bf solution} found by Molinet in \cite{molinetBOT} was to apply this estimate to $P_1$ 
 and the estimate indeed becomes uniform.  Our approach is quite in the same spirit however we believe that it gives another point of view and may be more 
 reliable in the case of damped equation or source term.

Let $N\in \N^*$ be large, to be chosen later. Applying the operator $P_{\geq N}$ to both sides to \eqref{inversionuw1} we obtain
\begin{equation*}
\begin{split}
P_{\geq N}u&=2iP_{\geq N}\big[we^{\frac{i}{2}F}\big]+P_{\geq N}\big[P_{\geq N}(e^{\frac{i}{2}F})P_{\leq 0}(ue^{-\frac{i}{2}F})\big]\\
&=A_N+B_N.
\end{split}
\end{equation*}
We apply the same estimates as before to get (the action of operator $P_{\geq N}$ on $L^4$ can be easily seen to be uniform on $N$ by noticing  that 
$P_{\geq N}=e^{iNx}P_+ e^{-iNx}$)
\begin{equation*}
\begin{split}
\|A_N \|_{L^{\infty}([0,T],L^2)}+\|A_N \|_{L^{4}([0,T]\times \T)}&\leq  C\,\big(\|w \|_{L^{\infty}([0,T],L^2)}+\|w \|_{L^{4}([0,T]\times \T)}\big)\\
&\leq C\,\big(\|w \|_{Z^{0,0}_T}+\|w \|_{X^{0,1/2}_T}\big)\leq C\,\|w \|_{X_2}.
\end{split}
\end{equation*}

For $B_N$, we also use Lemma \ref{lemmeexphigh} in the Appendix.
\begin{equation*}
\begin{split}
\|B_N&\|_{L^{\infty}([0,T],L^2)}+\|B_N\|_{L^{4}([0,T]\times \T)}\\
&\leq \|P_{\geq N}(e^{\frac{i}{2}F})\|_{L^{\infty}([0,T]\times \T)}\big(\|u\|_{L^{\infty}([0,T],L^2)}+\|u\|_{L^{4}([0,T]\times \T)}\big)\\
&\leq \frac{C}{\sqrt{N}}\|(u,w)\|_{X}^2.
\end{split}
\end{equation*}

The above estimates yield
\begin{lemme}\label{lemmaestimPNL2L4} For $N\in\N^*$ large, it holds that
\bnan
\label{estimPNL2L4}
\|P_{\geq N}u \|_{L^{\infty}([0,T],L^2)}+\|P_{\geq N}u \|_{L^{4}([0,T]\times \T)}\leq \|w \|_{X_2}+\frac{C}{\sqrt{N}}\|(u,w) \|_{X}^2.
\enan
\end{lemme}
\vskip.3mm
\subsubsection{Large data Lipschitz estimates}
The previous decomposition is still not sufficient. This time, it is because of the first term of the gauge transform. The difficulty now comes from the
 the low frequency of $F$  which a priori do not allow to make $P_{\geq N}\big[w_2\big(e^{\frac{i}{2}F_1}-e^{\frac{i}{2}F_2}\big)\big]$ small.

We change a little bit the decomposition inspired by Molinet \cite{molinetBOT} p. 663. Let $N\in \N^*$ be large,  to be chosen later.
\begin{equation*}
e^{-\frac{i}{2}F}=P_{\geq N}W+P_{\leq N}(e^{-\frac{i}{2}F})
\end{equation*}

Then decomposing $F=Q_{N}F+P_N F$, we get
\begin{equation*}
e^{-\frac{i}{2}Q_NF}=e^{\frac{i}{2}P_N F}\big(P_{\geq N}W+P_{\leq N}(e^{-\frac{i}{2}F})\big).
\end{equation*}

Taking derivative in $x$ yields
\begin{equation*}
(Q_Nu)e^{-\frac{i}{2}Q_NF}=-(P_Nu)e^{\frac{i}{2}P_NF}\big(P_{\geq N}W+P_{\leq N}(e^{-\frac{i}{2}F})\big)
+e^{\frac{i}{2}P_NF}\big(2iP_{\geq N}w+P_{\leq N}(ue^{-\frac{i}{2}F})\big)
\end{equation*}
or
\begin{equation*}
\begin{split}
Q_Nu&=-(P_Nu)e^{\frac{i}{2}P_NF}\big(P_{\geq N}W+P_{\leq N}(e^{-\frac{i}{2}F})\big)+e^{\frac{i}{2}P_NF}\big(2iP_{\geq N}w+P_{\leq N}(ue^{-\frac{i}{2}F})\big)\\
&\;\;\;+(Q_Nu)\big(1-e^{-\frac{i}{2}Q_NF}\big).
\end{split}
\end{equation*}

We now apply $P_{\geq 3N}$ to the last identity to obtain
\begin{equation*}
\begin{split}
P_{\geq 3N}u&=-P_{\geq 3N}\big[(P_Nu)e^{\frac{i}{2}P_NF}P_{\geq N}W\big]-P_{\geq 3N}\big[(P_Nu)e^{\frac{i}{2}P_NF}P_{\leq N}(e^{-\frac{i}{2}F})\big]\\
&\;\;\;+2iP_{\geq 3N}\big[e^{\frac{i}{2}P_NF}P_{\geq N}w\big]+P_{\geq 3N}\big[e^{\frac{i}{2}P_NF}P_{\leq N}(ue^{-\frac{i}{2}F})\big]\\
&\;\;\;+P_{\geq 3N}\big[(Q_Nu)\big(1-e^{-\frac{i}{2}Q_NF}\big)\big]\\
&=A_N+B_N+C_N+D_N+E_N.
\end{split}
\end{equation*}

We write 
\begin{equation*}
\begin{split}
A_{N,1}-A_{N,2}&=-P_{\geq 3N}\big[(P_N(u_1-u_2))e^{\frac{i}{2}P_N F_1}P_{\geq N}W_1+(P_Nu_2)\big(e^{\frac{i}{2}P_N F_1}-e^{\frac{i}{2}P_NF_2}\big)P_{\geq N}W_1\\
&\;\;\; +(P_Nu_2)e^{\frac{i}{2}P_NF_2}P_{\geq N}(W_1-W_2)\big].
\end{split}
\end{equation*}
Then using the Sobolev estimate \eqref{sobhigh} of  the Appendix on the terms with $W$, we get
\begin{equation}\label{AN1}
\begin{split}
\|A_{N,1}-&A_{N,2} \|_{L^{\infty}([0,T],L^2)}+\|A_{N,1}-A_{N,2} \|_{L^{4}([0,T]\times \T)}\\
&\leq \frac{C}{\sqrt{N}} \|U_1-U_2 \|_{X}\big[\|U_1 \|_{X}+\|U_1 \|_{X}^2\|U_1\|_{X}+\|U_1\|_X^2\|U_2 \|_{X}+\|U_2 \|_{X}^2\big]
\end{split}
\end{equation}

To obtain estimates on $B_N$, we first notice that due to the frequency localization, we have
$$
B_N=-P_{\geq 3N}\big[(P_Nu)P_{\geq N}(e^{\frac{i}{2}P_NF})P_{\leq N}(e^{-\frac{i}{2}F})\big].
$$
Thus we can write
\begin{equation*}
\begin{split}
B_{N,1}-B_{N,2} &=P_{\geq 3N}\big[(P_N(u_1-u_2))P_{\geq N}(e^{\frac{i}{2}P_NF_1})P_{\leq N}(e^{-\frac{i}{2}F_1})\\
&\;\;\;+(P_Nu_2)P_{\geq N}(e^{\frac{i}{2}P_NF_1}-e^{\frac{i}{2}P_NF_2})P_{\leq N}(e^{-\frac{i}{2}F_1})\\
&\;\;\;+(P_Nu_2)P_{\geq N}(e^{\frac{i}{2}P_NF_2})P_{\leq N}(e^{-\frac{i}{2}F_1}-e^{-\frac{i}{2}F_2})\big]
\end{split}
\end{equation*}

Using the estimates of Lemma \ref{lemmeexphigh} for the term with $P_{\geq N}$, we obtain
\begin{equation}\label{BN1}
\begin{split}
\|B_{N,1}&-B_{N,2}\|_{L^{\infty}([0,T],L^2)}+\|B_{N,1}-B_{N,2}\|_{L^{4}([0,T]\times \T)}\\
\leq &\frac{C}{\sqrt{N}} \|U_1-U_2 \|_{X}\big[\|U_1 \|_{X}+\|U_1 \|_{X}^2+\|U_1 \|_{X}^3+\|U_2 \|_{X}+\|U_2 \|_{X}^2+\|U_2\|_{X}^3\big].
\end{split}
\end{equation}

$C_N$ is the crucial term that contains the information on $w$ and we have to estimate it with respect to the norm of $P_N u$.
\begin{equation*}
C_{N,1}-C_{N,2}= 2iP_{\geq 3N}\big[\big(e^{\frac{i}{2}P_NF_1}-e^{\frac{i}{2}P_NF_2}\big)P_{\geq N}w_1-e^{\frac{i}{2}P_NF_2}P_{\geq N}(w_1-w_2)\big].
\end{equation*}
Thus
\begin{equation}\label{CN1}
\begin{split}
\|C_{N,1}-C_{N,2} \|_{L^{\infty}([0,T],L^2)}&+\|C_{N,1}-C_{N,2} \|_{L^{4}([0,T]\times \T)}\\
&\leq C\|P_N(u_n-u_2) \|_{L^{\infty}([0,T],L^2)}\|U_1 \|_{X}+C\|w_1-w_2 \|_{X_2}.
\end{split}
\end{equation}

$D_N$ can be treated similarly as $B_N$. By frequency localization, we also obtain
\begin{equation*}
D_N=P_{\geq 3N}\big[P_{\geq 2N}(e^{\frac{i}{2}P_NF})P_{\leq N}(ue^{-\frac{i}{2}F})\big]
\end{equation*}
and then
\begin{equation}\label{DN1}
\begin{split}
\|D_{N,1}-&D_{N,2} \|_{L^{\infty}([0,T],L^2)}+\|D_{N,1}-D_{N,2} \|_{L^{4}([0,T]\times \T)}\\
&\leq \frac{C}{\sqrt{N}} \|U_1-U_2 \|_{X}\big[\|U_1 \|_{X}+\|U_1 \|_{X}^2+\|U_1 \|_{X}^3+\|U_2 \|_{X}+\|U_2 \|_{X}^2+\|U_2 \|_{X}^3\big].
\end{split}
\end{equation}

For $E_N$, we use the third estimate of Lemma \ref{lemmeexphigh}.
\bna
E_{N,1}-E_{N,2}=P_{\geq 3N}\big[(Q_N(u_1-u_2))\big(1-e^{-\frac{i}{2}Q_NF_1}\big)\big]+P_{\geq 3N}\big[(Q_Nu_2)\big(e^{-\frac{i}{2}Q_NF_2}-e^{-\frac{i}{2}Q_NF_1}\big)\big].
\ena
Hence
\begin{equation}\label{EN1}
\begin{split}
\|E_{N,1}-E_{N,2} \|_{L^{\infty}([0,T],L^2)}&+\|E_{N,1}-E_{N,2} \|_{L^{4}([0,T]\times \T)}\\
&\leq \frac{C}{\sqrt{N}} \|U_1-U_2 \|_{X}\big[\|U_1 \|_{X}+\|U_1 \|_{X}^2+\|U_2 \|_{X}+\|U_2 \|_{X}^2\big].
\end{split}
\end{equation}

Gathering the estimates \eqref{AN1}-\eqref{EN1} results in the
\begin{lemme}\label{projLD}
\begin{equation}\label{Lipschitzlarghigh}
\begin{split}
\|P_{\geq 3N}&(u_1-u_2) \|_{L^{\infty}([0,T],L^2)}+\|P_{\geq 3N}(u_1-u_2) \|_{L^{4}([0,T]\times \T)}\\
&\leq C\|P_N(u_n-u_2) \|_{L^{\infty}([0,T],L^2)}\|U_1 \|_{X}+C\|w_1-w_2 \|_{X_2}\\
&\;\;\;+\frac{C}{\sqrt{N}} \|U_1-U_2 \|_{X}\big[\|U_1 \|_{X}+\|U_1 \|_{X}^2+\|U_1 \|_{X}^3+\|U_2 \|_{X}+\|U_2 \|_{X}^2+\|U_2 \|_{X}^3\big].
\end{split}
\end{equation}
\end{lemme}

\subsection{Large data global estimates}
From Lemmas \ref{estimatesw},   \ref{estimLipschilowA}, \ref{lemmaestimPNL2L4},  and  \ref{soluBO2} we have
\begin{equation*}
\begin{split}
&\|w \|_{X_2}\leq \|w_0 \|_{L^2}+CT^{\theta}\big( \|g \|_{L^{2}([0,T],L^2)}+\|g \|_{L^{2}([0,T],L^2)}^2+\|(u,w) \|_{X}^2+\|(u,w) \|_{X}^3\big),\\
&\|P_{\geq N}u \|_{L^{\infty}([0,T],L^2)}+\|P_{\geq N}u \|_{L^{4}([0,T]\times \T)}\leq \|w \|_{X_2}+\frac{C}{\sqrt{N}}\|(u,w) \|_{X}^2,\\
&\|u \|_{X^{-1,1}_T}\leq C\|u_0 \|_{L^2}+\| u \|_{L^4([0,T],L^4)}^2+\|g \|_{L^2([0,T],L^2)},
\end{split}
\end{equation*}
and
\begin{equation*}
\begin{split}
\|P_Nu \|_{L^{\infty}([0,T],L^2)}&+\|P_Nu \|_{L^{4}([0,T]\times \T)}\\
&\leq C(N) \|u_0 \|_{L^2}+C(N) \|g \|_{L^2([0,T]\times \T)}+C(N)T^{1/2}\|u \|_{L^{4}([0,T]\times \T)}^2.
\end{split}
\end{equation*}

Since the term $ \| u \|_{L^4([0,T],L^4)}^2$ in the third inequality has not factor of $T$ or $N^{-1/2}$  we need to combine the second and third estimates 
to obtain
\begin{equation*}
\begin{split}
\|u \|_{X_1}&\leq C(N)\big[\|u_0 \|_{L^2}+\|g \|_{L^2([0,T],L^2)}\big]+\|w \|_{X_2}+\frac{C}{\sqrt{N}}\|(u,w) \|_{X}^2+C(N)T^{1/2}\|u \|_{L^{4}([0,T]\times \T)}^2\\
&\;\;\;+\Big(C(N)\big[\|u_0 \|_{L^2}+\|g \|_{L^2([0,T],L^2)}\big]+\|w \|_{X_2}+\frac{C}{\sqrt{N}}\|(u,w) \|_{X}^2+C(N)T^{1/2}\|u \|_{L^{4}([0,T]\times \T)}^2\Big)^2.
\end{split}
\end{equation*}

Notice that $\|w_0 \|_{L^2}\leq \|u_0 \|_{L^2}$ and since we assume that $0\leq T\leq 1$, $N\in N^*$, we have
\begin{equation}\label{estimglobavant}
\begin{split}
\|(u,w) \|_{X}&\leq C(N)\big[\|u_0 \|_{L^2}+\|g \|_{L^2([0,T],L^2)}+\|g \|_{L^2([0,T],L^2)}^2\big]\\
&\;\;\;+\big(CT^{\theta}+\frac{C}{\sqrt{N}}+C(N)T^{1/2}\big)\big(\|(u,w) \|_{X}^2+\|(u,w) \|_{X}^3\big)\\
&\;\;\;+(C(N)\big[\|u_0 \|_{L^2}+\|g \|_{L^2([0,T],L^2)}+\|g \|_{L^2([0,T],L^2)}^2\big]\\
&\;\;\;+\big(CT^{\theta}+\frac{C}{\sqrt{N}}+C(N)T^{1/2}\big)\big(\|(u,w) \|_{X}^2+\|(u,w) \|_{X}^3)\Big)^2\\
&\leq C(N)\big[\|u_0 \|_{L^2}+\|u_0 \|_{L^2}^2+\|g \|_{L^2([0,T],L^2)}+\|g \|_{L^2([0,T],L^2)}^4\big]\\
&\;\;\;+\big(CT^{\theta}+\frac{C}{\sqrt{N}}+C(N)T^{1/2}\big)\big(\|(u,w) \|_{X}+\|(u,w) \|_{X}^6\big).
\end{split}
\end{equation}

With this estimate at hand we proceed to establish a uniform Lipschitz bound for large data.

We can now use a bootstrap argument (see for instance Lemma 2.2 of \cite{BahouriGerard}) using Lemma \ref{continuiteXt} for the continuity and limit in zero. 
Let $R>0$ such that $\|u_0 \|_{L^2}+\|g \|_{L^2([0,T],L^2)}\leq R$. We first pick $N$ large enough and then $T$ small enough (only depending on $R$ and universal constants) 
to get by a boot strap
\begin{equation*}
\|(u,w) \|_{X}\leq C(R)\big[\|u_0 \|_{L^2}+\|g \|_{L^2([0,T],L^2)}\big].
\end{equation*}
This gives the expected result for times $T$ small enough (the smallness only depending on $R$, i.e. $T\leq C(R)$). To globalize the result, we use the energy estimate
 (for $t$ positive or negative), obtained by multiplying the equation by $u$ and integration by parts, and which is valid for smooth solutions
\begin{equation*}
\|u(t) \|_{L^2}^2=\|u(0) \|_{L^2}^2-2\int_0^t  gu~ds.
\end{equation*}
By Cauchy-Schwarz inequality,
\begin{equation*}
\|u(t) \|_{L^2}^2\leq \|u(0) \|_{L^2}^2+\|g \|_{L^2([0,t],L^2)}^2+\frac{1}{2}\int_0^t  \|u(s) \|_{L^2}^2~ds.
\end{equation*}
We conclude by Gronwall's lemma that the $L^2$ norm is bounded on every compact interval.

\subsection{Large data Lipschitz estimates}

We assume $\|u_0 \|_{L^2}+\|g \|_{L^2([0,T],L^2)}\leq R$. From the previous section we have
\begin{equation*}
\|U \|_{X}\leq C(R).
\end{equation*}
Thus the estimates \eqref{estimwlip} and  \eqref{estimLipschilow} yield
\begin{equation}\label{largeA}
\|w_1-w_2 \|_{X_2}
\leq \|w_{0,1}-w_{0,2} \|_{L^2}+C(R)\|g_1-g_2 \|_{L^2([0,T]\times \T)}+C(R)T^{\theta}\|U_1-U_2 \|_{X}
\end{equation}
and
\begin{equation}\label{largeB}
\begin{split}
&\|P_N(u_1-u_2) \|_{L^{\infty}([0,T],L^2)}+\|P_N(u_1-u_2) \|_{L^{4}([0,T]\times \T)}\\
&\leq C(N) \|u_{0,1}-u_{0,2} \|_{L^2}+C(N) \|g_1-g_2 \|_{L^2([0,T]\times \T)}+C(N,R)T^{1/2}\|U_1-U_2 \|_{X}.
\end{split}
\end{equation}

The estimate \eqref{Lipschitzlarghigh} can be rewritten as
\begin{equation}\label{largeC}
\begin{split}
&\|P_{\geq N}(u_1-u_2) \|_{L^{\infty}([0,T],L^2)}+\|P_{\geq N}(u_1-u_2) \|_{L^{4}([0,T]\times \T)}\\
&\leq C(R)\|P_N(u_1-u_2) \|_{L^{\infty}([0,T],L^2)}+C\|w_1-w_2 \|_{X_2}+\frac{C(R)}{\sqrt{N}} \|U_1-U_2 \|_{X}.
\end{split}
\end{equation}
and from Lemma \ref{lipschequ} it follows that
\begin{equation}\label{largeD}
\|u_1-u_2 \|_{X^{-1,1}_T}\leq \|u_{0,1}-u_{0,2}\|_{L^2}+ \|g_1-g_2 \|_{L^{2}([0,T]\times \T)}+C(R)\| u_1-u_2 \|_{L^{4}([0,T]\times \T)}.
\end{equation}

Therefore, combining the estimates \eqref{largeA}-\eqref{largeD} leads to
\begin{equation*}
\begin{split}
\|U_1-U_2 \|_{X}&\leq C(N,R)\big(\|u_{0,1}-u_{0,2} \|_{L^2}+\|g_1-g_2 \|_{L^2([0,T]\times \T)}\big)\\
&\;\;\;+\big(C(N,R)T^{1/2}+\frac{C(R)}{\sqrt{N}}+C(R)T^{\theta} \big)\|U_1-U_2 \|_{X}.
\end{split}
\end{equation*}
So, by choosing $N$ large enough only depending on $R$ and then $T$ small enough (only depending on $N,R$ so only on $R$), we get
\begin{equation}\label{finalLipsch}
\|U_1-U_2 \|_{X}\leq C(N,R)\big(\|u_{0,1}-u_{0,2} \|_{L^2}+\|g_1-g_2 \|_{L^2([0,T]\times \T)}\big).
\end{equation}

The estimate \eqref{finalLipsch} allows us to establish the existence and uniqueness of the solution of the problem \eqref{sourceivp}. The continuity
of the flow follows by  Bona-Smith argument. This completes the proof of Theorem \ref{sourcelwp}.

\section{Control Results}

In this section we will prove Theorem \ref{control}.  First we will consider the damped equation and use the theory established
in the previous section to this case. 

\subsection{Case of the damped equation}

We consider the equation
\begin{equation}
\begin{cases}
\partial_t u+\H \partial_x^2 u=\frac{1}{2}\partial_x (u^2)-\G\G^{*}u\\
u(T)=0
\end{cases}
\end{equation}
We view the damping term as a source term, setting $g=-\G\G^{*}u$. Thus, we have 
$$
\|g \|_{L^2([0,T],L^2)}\leq CT^{1/2} \|u \|_{L^{\infty}([0,T],L^2)}\leq CT^{1/2}\|(u,w) \|_{X}.
$$ 

Hence estimate (\ref{estimglobavant}) becomes
\bna
\|(u,w) \|_{X}&\leq &C(N)\big[\|u_0 \|_{L^2}+\|u_0 \|_{L^2}^2\big]\\
&&+\big(CT^{\theta}+\frac{C}{\sqrt{N}}+C(N)T^{1/2}\big)\big(\|(u,w) \|_{X}+\|(u,w) \|_{X}^6\big).
\ena

We can conclude similarly by bootstrap for small times (only depending on the size). This gives the expected result for small times. For large times, 
we use the energy estimate (for $t$ positive or negative)
\bna
\|u(t) \|_{L^2}^2=\|u(0) \|_{L^2}^2-2\int_0^t  \|\G u(s) \|_{L^2}^2~ds.
\ena
For positive times, the energy is decreasing, for negative times, we conclude by Gronwall's  lemma.

We conclude similarly for the Lipschitz estimates.

Now we proceed to prove Theorem \ref{control}. We want to control the following problem
\begin{equation*}
\begin{cases}
\partial_t u+\H \partial_x^2 u=u\partial_x u-\G\G^{*}h,\\
u(T)=0
\end{cases}
\end{equation*}
where $\G$ is defined as  in \eqref{operG}.

We seek a control of the form of a solution of ($h_0$ is real valued with zero mean)
\begin{equation*}
\begin{cases}
\partial_t h+\H \partial_x^2 h=0\\
h(0)=h_0.
\end{cases}
\end{equation*}

We denote by B the nonlinear operator defined by $Bh_0=u(0)$.
Let $u_L$ be the solution of
\begin{equation*}
\begin{cases}
\partial_t u_L+\H \partial_x^2 u_L=-\G\G^{*}h\\
u_L(T)=0.
\end{cases}
\end{equation*}
The linear operator from $L^2(\T)$ to itself defined by $Lh_0=u_L(0)$ is the HUM operator, which is a bijection of $L^2(\T)$ by the observability inequality of Linares-Ortega \cite{LO}.
We note $Wh_0=u(0)=u_L(0)+Kh_0=Lh_0+Kh_0$. \refp{estimlinaru} can be written as
\begin{equation*}
\|Kh_0 \|_{L^2}\leq C \|h_0 \|_{L^2}^2
\end{equation*}
for $\|h_0 \|_{L^2}$ small enough.
Let $u_0\in L^2(\T)$. So, the objective $Wh_0=Lh_0+Kh_0=u_0$ is equivalent to  $h_0=-L^{-1}Kh_0+L^{-1}u_0$, that is $h_0$ is a fixed point of $B$ defined by $Bh_0=-L^{-1}Kh_0+L^{-1}u_0$.

But, we have 
$$
\|Bh_0 \|_{L^2}\leq C\|Kh_0 \|_{L^2}+C\|u_0 \|_{L^2}\leq C \|h_0 \|_{L^2}^2+C\|u_0 \|_{L^2}.
$$
Thus, if we denote $B_R$ the unit ball of $L^2$ of radius $R=2C\|u_0 \|_{L^2}$, $B$ sends $B_R$ into itself if $CR\leq 1/2$, that is,  if $\|u_0 \|_{L^2}$ is chosen small enough.

Moreover, estimate \refp{estimLipfinkin} shows that for two solutions of the nonhomogeneous problem $u_1$, $u_2$ coming from the controls $g_i=-\G\G^{*}h_i$ 
(recall that we have assumed $\|g_i \|_{L^{2}([0,T]\times \T)}\leq \e$, which is equivalent to $\|h_{0,i} \|_{L^2}\lesssim \e$), we have the estimates
\bna
\|Kh_{0,1}-Kh_{0,1} \|_{L^2}\leq \|u_1-u_{L,1}-(u_2-u_{L,2}) \|_{L^{\infty}([0,T],L^2)}
&\lesssim &\e\|g_1-g_2 \|_{L^2([0,T]\times \T)}.
\ena
Since $ Bh_{0,1}-Bh_{0,2}=Kh_{0,1}-Kh_{0,2}$, this means that for $\e$ small enough (that is $R=2C\|u_0 \|_{L^2}$ small enough), $B$ is contracting and reproduce $B_R$. Therefore, it has a fixed point which is the expected control.

It ends the proof of Theorem \ref{control} part (i) for small data.

\section{Study of contradicting bounded sequences of damped equations}

\subsection{Large data}
In this section $u_n$ is a solution as in  Lemma \ref{lemmacvceweak}. We assume to be proved that $(u_n, w_n)$ is bounded in $X$. The final aim is to prove the strong convergence in $L^2$ as stated in Theorem \ref{lemmacvcestrong}. We assume $[u_0]=0$ and explain the required modifications in a remark for the general case. Mainly, according to Remark \ref{rkchgevarstab}, the modifications require to apply the same arguments to $\widetilde{u}$ and to replace $\G$ by $\G_{\mu}$.

\begin{lemme}
\label{lemmacvceweak}
Suppose $u_{n,0}$ is a bounded sequence in $L^2$ with $[u_{0,n}]=0$ with associated solutions $u_n$
\bneq
\partial_t u_n+\H \partial_x^2 u_n&=&u_n\partial_x u_n(-\G\G^{*}u_n)\\
u_n(0)&=&u_{0,n}.
\eneq
Assume moreover that $\G u_n\rightarrow 0$ in $L^2([0,T]\times \T)$. Then, $u_n\rightharpoonup 0$ in $L^{2}([0,T],L^2)$ and weakly-* in $L^{\infty}([0,T],L^2)$.
\end{lemme}
\begin{remark}
We could certainly get rid of the term $\G\G^{*} u_n$ (converging to $0$ in $L^2([0,T]\times L^2$)) and work with solutions of the free equation. Yet, we have chosen to keep it since it 
did not perturb too much the analysis and it seemed easier than proving a general perturbation theorem.
\end{remark}
\bnp
We assume $u_n\rightharpoonup u$ in $L^2([0,T],L^2)$ up to subsequence. The only problem to use directly Proposition 2.8 of \cite{LinaresRosierBO} is that because of the nonlinear term, it is not clear that $u$ is solution of the Benjamin-Ono equation (see Molinet \cite{Molinet,MolinetBOill} for some cases where the limit equation is a modified equation).

Denote the sequence $c_n(t)=\int_{\T}a(y)u_n(x,t)dy$, bounded in $L^2([0,T])$ and weakly convergent to $c(t)$. Denote $r_n=u_n-c_n$. We have $a(x)\big[r_n\big]\rightarrow 0$, in $L^2([0,T]\times \T)$. In particular, $r_n\rightarrow 0$ in $L^2([0,T]\times \omega)$ and $r_n^2\rightarrow 0$ in $L^1([0,T]\times \omega)$. Therefore, $r_n^2\rightharpoonup 0$ in the distributional sense of $D'(]0,T[\times \omega)$ and hence it is the same for $\partial_x r_n^2$. But we have
\bna
\partial_x r_n^2&=&\partial_x u_n^2-2\partial_x (c_nu_n)+\partial_x c_n^2=\partial_x u_n^2-2\partial_x (c_nu_n)\\
&=&\partial_x u_n^2-2\partial_x (c_nr_n)
\ena
But, we have  $r_n\rightarrow 0$ in $L^2([0,T]\times \omega)$ so $c_nr_n\rightarrow 0$ in $L^1([0,T]\times \omega)$ and so in the distributional sense in $D'(]0,T[\times \omega)$, and the same for $\partial_x (c_nr_n)$.

So, we conclude that $\partial_x u_n^2\rightharpoonup 0$ in the distributional sense of $D'(]0,T[\times \omega)$. In particular, the weak limit $u$  is solution in the sense of distribution $D'(]0,T[\times \omega)$ of
\bna
u_{t}+\H u_{xx}=\dot{c}(t)+\H u_{xx}=0.
\ena
Therefore, we have $\H u_{xxx} =0$ in the distributional sense of $D'(]0,T[\times \omega)$.

Moreover, since $r_n$ converges to zero in the distributional sense of $D'(]0,T[\times \omega)$, it is the same for $r_{n,xxx}=u_{n,xxx}$ and we have $u_{xxx} =0$ in the distributional sense of $D'(]0,T[\times \omega)$.

We conclude that $u\equiv 0$ as in Linares-Rosier \cite{LinaresRosierBO} Proposition 2.8.
\enp

\begin{remark}
In the case of $[u_{0,n}]=\mu\neq 0$, we can, for instance, set $\underline{u}_n=u_n-\mu$ and obtain similarly $u_n \rightharpoonup \mu$ using that $\underline{u}_n$ is solution of $\partial_t \underline{u}_n+\H \partial_x^2 \underline{u}_n=\mu\partial_x \underline{u}_n+ \underline{u}_n\partial_x \underline{u}_n-\G\G^{*}\underline{u}_n$.
\end{remark}

\begin{lemme}
\label{lmweakwn}
$u_n\rightharpoonup 0$ in $L^2([0,T],L^2)$ implies $w_n\rightharpoonup 0$ in $L^2([0,T],L^2)$ with $w_n=P_+(u_ne^{-\frac{i}{2}F_n})$ (for solutions). The same holds for $u_ne^{-\frac{i}{2}F_n}$.
\end{lemme}
\bnp
$F_n$ weakly-* converges to $0$ in $L^{\infty}([0,T],H^1)$ and $\partial_t F_n$ is bounded in $L^{\infty}([0,T],H^{-1})$ and hence by Aubin-Lions lemma, $F_n$ converges strongly to $0$ in $L^{\infty}([0,T],L^{\infty})$. By the mean value theorem, $e^{-\frac{i}{2}F_n}$ converges strongly to $1$ in $L^{\infty}([0,T],L^{\infty})$. Actually, $\partial_t F_n=-\H \partial_x^2 F_n+\frac{1}{2} (\partial_x F_n)^2-\frac{1}{2}P_0 (F_{n,x}^2)+\partial_x^{-1}\G\G^{*} u_n$ is bounded in $L^{\infty}([0,T],H^{-1})$ since $\partial_x^2 F_n=\partial_x u_n$ and $(\partial_x F_n)^2=u_n^2$ is bounded in $L^{\infty}([0,T],L^{1})\subset L^{\infty}([0,T],H^{-1})$.

Let $\varphi\in L^2([0,T],L^2)$. We write
\bna
\iint_{[0,T]\times \T}w_n \overline{\varphi}=\iint_{[0,T]\times \T}u_ne^{-\frac{i}{2}F_n} P_+\overline{\varphi}=\iint_{[0,T]\times \T}u_n(e^{-\frac{i}{2}F_n}-1) P_+\overline{\varphi}+\iint_{[0,T]\times \T}u_n P_+\overline{\varphi}.
\ena
The first term converges to $0$ by strong convergence of $e^{-\frac{i}{2}F_n}-1$ in $L^{\infty}([0,T],L^{\infty})$. The second converges to $0$ by weak convergence. The same proof without $P_+$ gives the result for $u_ne^{-\frac{i}{2}F_n}$.
\enp
\begin{lemme}\label{lmcvcecn}
$c_n(t)=\int a(y)u_n(y,t)dy$ is bounded in $H^{1}([0,T])$ and converges strongly to $0$ in $L^2([0,T])$.
\end{lemme}
\bnp
We compute
\bna
\dot{c_n}(t)&=&\int_{\T} a(y)\big[\H \partial_{y}^2 u_n(t,y)+\frac{1}{2}\partial_y(u_n(t,y)^2)+\G\G^{*} u_n\big]dy\\
&=&-\int_{\T} \H \partial_{y}^2a(y) u_n(t,y)dy-\frac{1}{2}\int_{\T} \partial_ya(y)(u_n(t,y)^2)dy+\int_{\T} a(y)\G\G^{*} u_ndy.
\ena
\bna
\big|\dot{c_n}(t)\big|&\leq &C\|u_n(t) \|_{L^2(\T)} +C\|u_n(t) \|_{L^2(\T)}^2
\ena
So $\|\dot{c_n}(t) \|_{L^2(]0,T[)}\leq C\|u_n \|_{L^2([0,T],L^2)}+C\|u_n \|_{L^{\infty}([0,T],L^2)}^2\leq C$. Lemma \ref{lmcvcecn} follows  by Sobolev embedding and weak 
convergence to $0$ of $u_n$.
\enp
\begin{remark}
For Lemma \ref{lmweakwn} and \ref{lmcvcecn}, the same argument applies with $\G$ replaced by $\G_{\mu}$ since we only use the boundedness as an operator of $L^{\infty}([0,T],L^2)$.
\end{remark}

\begin{lemme}
\label{lmundonnewn}
$\G u_n\rightarrow 0$ in $L^2([0,T]\times \T)$ implies that $a(x)w_n\rightarrow 0$ in $L^2([0,T]\times \T)$.
\end{lemme}
\bnp
Lemma \ref{lmcvcecn} implies that we have in fact $a(x)u_n\rightarrow 0$ in $L^2([0,T]\times \T)$ and so $a(x)e^{-\frac{i}{2}F_n}u_n\rightarrow 0$ in $L^2([0,T]\times \T)$, and the same for $P_+(a(x)e^{-\frac{i}{2}F_n}u_n)$.

We write
\bna
a(x)w_n=a(x)P_+(u_ne^{-\frac{i}{2}F_n})=P_+(a(x)u_ne^{-\frac{i}{2}F_n})+[a(x),P_+]\big(u_ne^{-\frac{i}{2}F_n}\big).
\ena
So, it remains to prove that the second term in the right hand side is strongly convergent in $L^2([0,T]\times \T)$.

Using Lemma 2.5 of Linares-Rosier \cite{LinaresRosierBO}, we get
\begin{equation*}
\|[a(x),P_+]\big(u_ne^{-\frac{i}{2}F_n}\big)\|_{L^2([0,T]\times \T)}\leq \|u_ne^{-\frac{i}{2}F_n} \|_{L^2([0,T],H^{-1})}.
\end{equation*}

Using Lemma \ref{lmweakwn}, we get that $u_ne^{-\frac{i}{2}F_n}\rightharpoonup 0$ in $L^2([0,T]\times \T)$. To apply Aubin-Lions' lemma, we just need to prove that
 $\partial_t \big(u_ne^{-\frac{i}{2}F_n}\big)$ is bounded in some space $L^p([0,T],H^{-s})$, $s\in \R$. We also notice that $u_ne^{-\frac{i}{2}F_n}=2i\partial_x e^{-\frac{i}{2}F_n}$, so , we just have to prove that $\partial_t e^{-\frac{i}{2}F_n}=-\frac{i}{2}(\partial_t F_n )e^{-\frac{i}{2}F_n}$ is bounded in some $L^p([0,T],H^{-s})$. But, we have shown  in the proof of Lemma \ref{lmweakwn} that $\partial_t F_n$ was bounded in $L^{\infty}([0,T],H^{-1})$ and we notice easily that $e^{-\frac{i}{2}F_n}$ is bounded in $L^{\infty}([0,T],H^{1})$, since $\partial_x e^{-\frac{i}{2}F_n}=-\frac{i}{2}u_ne^{-\frac{i}{2}F_n}$ is bounded in $L^{\infty}([0,T],L^2)$. Thus by product, $\partial_t e^{-\frac{i}{2}F_n}$ is bounded in $L^{\infty}([0,T],H^{-1})$.
\enp
\begin{remark}
Lemma \ref{lmundonnewn} has to be changed when $\G$ is replaced by $\G_{\mu}$. Indeed, in that case, the zone of damping is moving at speed $\mu$. But, we can use the infinite speed of propagation to get a similar result as follows.

If $\G_{\mu} \widetilde{u}_n\rightarrow 0$ in $L^2([0,T]\times \T)$, we infer similarly that $a(x-t\mu)\widetilde{u}_n\rightarrow 0$ in $L^2([0,T]\times \T)$. If for instance $[-c,c] \subset \omega$ with $c>0$, we easily conclude that there exists an $\e$ (a priori depending on $\mu$) such that $\widetilde{u}_n\rightarrow 0$ in $L^2([0,\e]\times [-\e,\e])$.  

This result will be enough for what follows since the time $T>0$ and the open set $\omega$ are arbitrary ($\omega$ non empty).
\end{remark}

Now, we borrow a propagation theorem from \cite{laurent}.
\begin{theorem}
\label{thmpropagation3}
Let $w_n$ be a sequence of solutions of
$$i\partial_t w_n + \partial_x^2 w_n =f_n$$
such that for one $0\leq b\leq 1$, we have
$$ \big\|w_n\big\|_{X^{0,b}_T} \leq C,~~\big\|w_n\big\|_{X^{-1+b,-b}_T} \rightarrow 0~~ and~~\big\|f_n\big\|_{X^{-1+b,-b}_T}\rightarrow  0 $$
Moreover, we assume that there is a non empty open set $\omega$ such that $w_n \rightarrow 0$ in $L^2([0,T],L^2(\omega))$.\\
Then $w_n \rightarrow 0$ in $L^2_{loc}([0,T],L^2(\T))$.
\end{theorem}
We want to apply this lemma with $b=1/2$ and $$f_n=-\partial_x P_+\big[W_n(  P_- \partial_x u_n) \big]-\frac{i}{2}P_+\big[g_n e^{-\frac{i}{2}F_n}\big]-\frac{1}{4}P_+\big[\big(-\frac{1}{2}P_0 (u_n^2)+G_n\big) u_ne^{-\frac{i}{2}F_n}\big]$$
with $g_n=\G\G^{*}u_n$, $G_n=\partial_x^{-1}g_n$.
That is, we need to prove
$$ \big\|w_n\big\|_{X^{0,\frac{1}{2}}_T} \leq C,~~\big\|w_n\big\|_{X^{-\frac{1}{2},-\frac{1}{2}}_T} \rightarrow 0~~ and~~\big\|f_n\big\|_{X^{-\frac{1}{2},-\frac{1}{2}}_T}\rightarrow  0. $$
We decompose $f_n=f_n^a+f_n^b$ with $f_n^a=-\partial_x P_+\big[W_n(  P_- \partial_x u_n) \big]$.
We know that
\begin{itemize}
\item $\|w_n \|_{X^{0,1/2}}\leq C$,
\item $\|f_n \|_{X^{0,-1/2}}\leq C$.
\end{itemize}

\begin{lemme}For any $\theta\in[0,1]$, we have the bound
\bna
\|f_n \|_{X_T^{-2\theta,-\frac{1-\theta}{2}}}\leq C.
\ena
\end{lemme}
\bnp
If  we prove $\|f_n \|_{X_T^{-2,0}}= \|f_n \|_{L^2([0,T],H^{-2})}\leq C$, by interpolation, this gives
\bna
\|f_n \|_{X_T^{-2\theta,-\frac{1-\theta}{2}}}\le C.
\ena
So, to bound $f_n^a$, we need to prove that $W_n(  P_- \partial_x u_n) $ is bounded in $L^2([0,T],H^{-1})$. We write
\begin{equation*}
\begin{split}
\|W_n(  P_- \partial_x u_n) \|_{L^2([0,T],H^{-1})}&\leq C \|W_n \|_{L^2([0,T],H^{1})}\|\partial_x u_n \|_{L^{\infty}([0,T],H^{-1})}\\
&\leq C \|w_n \|_{L^2([0,T],L^2)}\|u_n \|_{L^{\infty}([0,T],L^2)}\leq C
\end{split}
\end{equation*}
where we have used that $H^1$ is an algebra and so by duality, $H^1 *H^{-1}\hookrightarrow H^{-1}$.

The bound for $f_n^b$ is easier (and corresponds actually to the estimates of $II$ and $III$ in section \ref{subsectwestimgauge})
\bna
\|f_n^b \|_{L^2([0,T],H^{-2})}&\leq &\|f_n^b \|_{L^2([0,T],L^2)} \leq \|g_n \|_{L^2([0,T],L^2)}+\|-\frac{1}{2}P_0 (u_n^2)+G_n \|_{L^\infty([0,T]\times \T)}\|u_n \|_{L^2([0,T],L^2)}\\
&\leq &\|g_n \|_{L^2([0,T],L^2)}+\big(\|u_n \|_{L^{\infty}([0,T],L^2)}^2+\|g_n \|_{L^2([0,T],L^2)}\big)\|u_n \|_{L^2([0,T],L^2)}.
\ena
\enp
With for instance $\theta=\frac{1}{5}$,
\bna
\|f_n \|_{X_T^{-\frac{2}{5},-\frac{4}{10}}}\leq C
\ena
So, since $-\frac{1}{2}<-\frac{2}{5}$ and  $-\frac{1}{2}<-\frac{4}{10}$, the embedding of $X_T^{-\frac{2}{5},-\frac{4}{10}} \hookrightarrow X_T^{-\frac{1}{2},-\frac{1}{2}}$ is compact 
and we can extract a subsequence and pick a  function $f$ such that
\bna
\|f_n-f \|_{X_T^{-\frac{1}{2},-\frac{1}{2}}}\rightarrow 0.
\ena
We can get that $f=0$ using the equation verified by $w_n$, $f_n$ and use the fact that we already know that $w_n\rightharpoonup 0$ in the distributional sense.
\bna
\|f_n \|_{X_T^{-\frac{1}{2},-\frac{1}{2}}}\rightarrow 0.
\ena
This constitutes the case $b=\frac{1}{2}$. The same holds for $w_n$.

\medskip

We can now state the following theorem.

\begin{theorem}
\label{lemmacvcestrong}
Under the same assumptions as in Lemma \ref{lemmacvceweak}, we have $u_n \rightarrow 0$ in $L^2_{loc}([0,T],L^2(\T))$
\end{theorem}
\bnp
Using Theorem \ref{thmpropagation3} and the previous analysis, we get $w_n \rightarrow 0$ in $L^2_{loc}([0,T],L^2(\T))$. By definition of $w_n$, we have 
$w_n=-\frac{i}{2}P_+(u_ne^{-\frac{i}{2}F_n})$. But, we have proved in Lemma \ref{lmweakwn} that $e^{-\frac{i}{2}F_n}$ converges strongly to $1$ in $L^{\infty}([0,T],L^{\infty})$. 
So, we can write
\bna
P_+u_n=2iw_n-P_+\big[u_n(e^{-\frac{i}{2}F_n}-1)\big].
\ena
This gives $P_+u_n \rightarrow 0$ in $L^2_{loc}([0,T],L^2(\T))$ and the same result for $u_n$ since $u_n$ is real valued.
\enp
\subsection{Small data} 
\begin{lemme}\label{lemmaapproxlindamped}
For  $u_{0}\in L^2(\mathbb T)$, let $u$ denote  the solution of  
\begin{equation*}
\begin{cases}
\partial_t u+\H \partial_x^2 u =u\partial_x u-\G\G^{*} u,\\
u(0)=u_{0},
\end{cases}
\end{equation*}
and let  $u_{L}$ denote the solution of
\begin{equation}\label{lineq20}
\begin{cases}
\partial_t u_{L}+\H \partial_x^2 u_{L}= -\G\G^{*} u_{L}\\
u_{L}(0)=u_{0}.
\end{cases}
\end{equation}
Then, there exist some constants $T_0>0$ , $\epsilon>0$ and $C>0$ such that for $T<T_0$ and $\|u_{0} \|_{L^2}\leq \e$, we have
\begin{equation*}
\|u-u_L \|_{L^{\infty}([0,T],L^2)}\leq C\, \|u_{0} \|_{L^2}^2
\end{equation*}
\end{lemme}

\begin{proof}

Let $v$ denote the solution of 
\begin{equation}\label{lineqv}
\begin{cases}
\partial_t v+\H \partial_x^2 v= -\G\G^{*} u\\
v(0)=u_{0}.
\end{cases}
\end{equation}

From \eqref{estimlinaru}  applied with $g=-\G\G^{*}u$ and $\|u_0\|_{L^2}$ small enough, we have 
\begin{equation}\label{lineqv1}
\begin{split}
\|u-v\|_{L^{\infty}([0,T],L^2)}&\lesssim\|\G\G^{*}u\|_{L^2([0,T],L^2)}^2+\|u_0\|^2_{L^2} \\
&\lesssim \|u\|^2_{L^2([0,T],L^2)}+\|u_0\|^2_{L^2}\\
&\lesssim C_1(T+1)\,\|u_0\|^2_{L^2}
\end{split}
\end{equation}
for some constant $C_1>0$ where we used the fact that $\|u(t)\|_{L^2}\le \|u_0\|_{L^2}$ for any $t\ge 0$. 

On the other hand,  by classical semigroup estimates we have that 
\begin{equation}\label{lineqv2}
\begin{split}
\|v-u_L\|_{L^{\infty}([0,T],L^2)}&\lesssim \|\G\G^{*}(u-u_L)\|_{L^1([0,T],L^2)}\\
&\lesssim \|u-u_L\|_{L^1([0,T],L^2)}\\
&\le C_2\,T\,\|u-u_L\|_{L^{\infty}([0,T],L^2)}.
\end{split}
\end{equation}

Combining \eqref{lineqv1} and \eqref{lineqv2} and the triangle inequality yield
\begin{equation*}
\|u-u_L\|_{L^{\infty}([0,T],L^2)}\le C_1\,\|u_0\|^2_{L^2}+ C_2\,T\,\|u-u_L\|_{L^{\infty}([0,T],L^2)}.
\end{equation*}

Thus 
\begin{equation*}
\|u-u_L\|_{L^{\infty}([0,T],L^2)}\le\frac{C_1\, T}{1-C_2\,T}\,\|u_0\|^2_{L^2}
\end{equation*}
whenever $\|u_0\|_{L^2}<\epsilon$ and $T<T_0=\frac{1}{2C_2}$.

\end{proof}

\begin{prop}
\label{lemmacontrad0}
Suppose $u_{n,0}$ is a sequence of smooth functions strongly convergent to $0$ in $L^2$, with associated solutions $u_n$ of the problem
\begin{equation*}
\begin{cases}
\partial_t u_n+\H \partial_x^2 u_n=u_n\partial_x u_n-\G\G^{*} u_n\\
u_n(0)=u_{0,n}, 
\end{cases}
\end{equation*}
and in addition assume that
\bnan
\label{hypcontrad}
\|\G u_n \|_{L^2([0,T]\times \T)}\leq\frac{1}{n}\|u_{0,n} \|_{L^2}.
\enan

Then, $u_{0,n}=0$ for $n$ large enough.
\end{prop}
\bnp
We denote $u_{n,L}$ the solution of
\begin{equation*}
\begin{cases}
\partial_t u_{n,L}+\H \partial_x^2 u_{n,L}=-\G\G^{*} u_{n,L}\\
u_{n,L}(0)=u_{0,n}.
\end{cases}
\end{equation*}
By the observability of the linear damped system, proved in \cite{LO} Theorem 1.2, we know that for some $T$ that can be chosen $T<T_0$, we have
\bnan
\label{observdampedlin}
\|u_{0,n}\|_{L^2}^2\leq C\int_0^T \|\G u_{n,L}\|_{L^2}^2.
\enan
But, we have from triangular inequality, (\ref{hypcontrad}) and Lemma \ref{lemmaapproxlindamped}, that
\bna
\|\G u_{n,L}\|_{L^2([0,T]\times \T)}^2&\leq &C  \|\G u_{n}\|_{L^2([0,T]\times \T)}^2+C \|\G(u_{n,L}-u_n)\|_{L^2([0,T]\times \T)}^2\\
&\leq &\frac{C}{n}\|u_{0,n}\|_{L^2}^2+C\|u_{0,n}\|_{L^2}^4
\ena
Combining the previous estimate with (\ref{observdampedlin}) gives
\bna
\|u_{0,n}\|_{L^2}^2\leq C\big(\frac{1}{n}+\|u_{0,n}\|_{L^2}^2\big)\|u_{0,n}\|_{L^2}^2
\ena
Since $\|u_{0,n} \|_{L^2}$ converges to zero, this gives $u_{0,n}=0$ for $n$ large enough.
\enp

\section{Stabilization}

In this section we will prove Theorem \ref{stabilization}.  We will use the argument in \cite{LinaresRosierBO}. However to complete
the arguments we have to make some modifications as we  did in the previous section.

Because of the identity
\begin{equation}\label{energynew}
\frac12\|u(t)\|^2+\|\G u\|^2_{L^2([0,T]\times\T)}=\frac12\|u_0\|^2
\end{equation}
 we observe that $\|u(t)\|$ is nonincreasing, so that the exponential decay is guaranteed
if 
$$
\|u((n+1)T)\|\le \kappa\|u(nT)\|\text{\hskip10pt  for some \hskip10pt} \kappa<1.
$$

To prove the theorem, it is sufficient to show the following observability
inequality:  For any $T>0$ and any $R\gg1$ there exists a constant $C(R,T)$ such that for any $u_0\in H^0_0(\T)$ with $\|u_0\|\le R$
it holds
\begin{equation}\label{observanew}
\|u_0\|^2\le C\,\int_0^T\|\G u(t) \|^2\,dt
\end{equation}
where  $u$ denotes the solution of \eqref{stabilization1}. 

Suppose there exists a sequence $u_{0,n}\in H^0_0(\T)$ such that for each $n$ we have $ \|u_{0,n}\|\le R$ (where $R\gg1$) but
\begin{equation}
\label{contrad}
\|u_{0,n}\|>n\int_0^T\|\G u_n(t)\|^2\,dt.
\end{equation}
Then from Theorem \ref{lemmacvcestrong}, we have for a subsequence (still denoted by $u_n$) that
\begin{equation*}
u_n\to 0 \text{\hskip10pt in \hskip10pt} L^2_{\rm loc}([0,T], L^2(\T)).
\end{equation*}
Thus one can select some $t_0\in (0,T)$ such that, extracting again a subsequence,
\begin{equation*}
u_n(t_0)\to 0 \text{\hskip10pt in \hskip10pt} L^2(\T).
\end{equation*}

Since 
\begin{equation*}
\frac12\|u_n(t_0)\|^2+\|\G u_n\|^2_{L^2([0,t_0]\times\T)}=\frac12\|u_{0,n}\|^2,
\end{equation*}
we conclude that 
\begin{equation*}
\|u_{0,n}\|_{L^2}\to 0.
\end{equation*}
The assumptions of Proposition \ref{lemmacontrad0} are fulfilled. We finally get $u_{0,n}=0$ for $n$ large enough, which is a contradiction to \eqref{contrad}.


\section*{Acknowledgments} 
CL was partially supported by the Agence Nationale de la Recherche (ANR), Projet Blanc EMAQS number ANR-2011-BS01-017-01. CL would also like to thank IMPA for its hospitality. FL was partially supported by FAPERJ and CNPq Brazil.  LR was partially supported by the Agence Nationale de la Recherche (ANR), Project CISIFS,  grant ANR-09-BLAN-0213-02. 

\appendix

\section{\hskip10pt}

In the following,  we will establish a slight modification of the bilinear estimate of Molinet-Pilod \cite{MolinetPilodBO}, Proposition 3.5, with a gain of a power of $T$. 
Such an estimate will  allow to avoid the dilation argument for small time and large data and it could also be interesting for other purposes. This type of gain was already
obtained by Bourgain in  \cite{BourgainKdV} for the  KdV equation  and it relies mainly on the fact that the equation we consider is subcritical on $L^2$ with respect to the 
scaling.
\begin{lemme}
We have for $0\leq T\leq 1$ the estimate
\bna
\|\partial_x P_+\big[W(  P_- \partial_x u) \big] \|_{X^{0,-1/2}_T}\leq CT^{1/8} \|(u,w) \|_{X}^2.
\ena
\end{lemme}
\bnp
Let $h$, $u$, $w$ be some extension (same notation by abuse).
We use the same notation as in \cite{MolinetPilodBO} and continue the computations started in \cite{MolinetPilodBO}. We continue the estimates of $I_A$ starting from estimate (3.32) in \cite{MolinetPilodBO}.
\bna
|I_A|\leq \|h \|_{L^2}\|w \|_{X^{0,3/8}_T}\|u \|_{L^4}\leq CT^{1/2-3/8}\|h \|_{L^2}\|w \|_{X^{0,1/2}_T}\|u \|_{L^4},
\ena
For $I_B$, we use (3.33) and (3.34) of \cite{MolinetPilodBO} to get
\begin{equation*}
\begin{split}
|I_B| &\leq \big(\sum_{N_1}\|P_{N_1}\big(\frac{\widehat{h}}{\big\langle \sigma \big\rangle^{1/2}}\big)^{\vee} \|_{L^4}^2\big)^{1/2} \|w\|_{X^{0,1/2}_T}\|u \|_{L^4}\\
&\leq CT^{1/2-3/8}\|\big(\frac{\widehat{h}}{\big\langle \sigma \big\rangle^{1/2}}\big)^{\vee}\|{~X^{0,1/2}_T}\|h \|_{L^2}\|w \|_{X^{0,3/8}_T}\|u \|_{L^4}\\
&\leq CT^{1/2-3/8}\|h \|_{L^2}\|w \|_{X^{0,3/8}_T}\|u \|_{L^4},
\end{split}
\end{equation*}
and using (3.38)
\bna
|I_C|\leq \big(\sum_{N_1}\|P_{N_1}\big(\frac{\widehat{h}}{\big\langle \sigma \big\rangle^{1/2}}\big)^{\vee} \|_{L^4}^2\big)^{1/2}\|w \|_{X^{0,3/8}_T}\|u \|_{X^{-1,1}_T}
\leq CT^{1/2-3/8}\|h \|_{L^2}\|w \|_{X^{0,1/2}_T}\|u \|_{X^{-1,1}_T}.
\ena
\enp
\begin{lemme}
Let $\e>0$. The following estimate holds uniformly for $0\leq T\leq 1$:
\bna
\|\partial_x P_+\big[W(  P_- \partial_x u) \big] \|_{\widetilde{Z}^{0,-1}_T}\leq CT^{1/8-\e} \|(u,w) \|_{X}^2.
\ena
\end{lemme}
\bnp
We can see using (3.43) of \cite{MolinetPilodBO} that
\bna
|J_A|\leq \big(\sum_{N}\|g_N \|_{L^2_{\xi}L^{\infty}_{\tau}}^2\big)^{1/2}\|w \|_{X^{0,3/8}_T}\|u \|_{L^4}
\leq CT^{1/2-3/8}\big(\sum_{N}\|g_N \|_{L^2_{\xi}L^{\infty}_{\tau}}^2\big)^{1/2}\|w \|_{X^{0,1/2}_T}\|u \|_{L^4}
\ena
and using estimate line 15 p 381  of \cite{MolinetPilodBO}
\begin{equation*}
|J_B|+|J_C|\leq \Big(\Big\|\Big(\frac{g}{\langle \sigma\rangle}\Big)^{\vee}\Big \|_{\widetilde{L}^4}
+\Big\|\Big(\frac{|g|}{\langle \sigma\rangle}\Big)^{\vee}\Big\|_{\widetilde{L}^4}\Big)\|w\|_{X^{0,1/2}_T}\big(\|u\|_{L^4}+\|u\|_{X^{-1,1}_T}\big).
\end{equation*}
But
\begin{equation*}
\begin{split}
\Big\|\Big(\frac{g}{\langle \sigma\rangle}\Big)^{\vee}\Big\|_{\widetilde{L}^4}
+\Big\|\Big(\frac{|g|}{\langle \sigma\rangle}\Big)^{\vee}\Big \|_{\widetilde{L}^4}
&\leq CT^{1/2-3/8-\e}\Big\|\Big(\frac{|g|}{\langle \sigma \rangle}\Big)^{\vee}\Big\|_{X^{0,1/2-\e}_T}\\
&\leq  CT^{1/8-\e}\Big\|\frac{|g|}{\langle \sigma \rangle^{1/2+\e}}\Big \|_{L^2_{\tau}L^2_{\xi}}\\
&\leq CT^{1/8-\e}\big(\sum_N \|\big\langle \sigma \big\rangle^{-1/2-\e}g_N \|_{L^2_{\tau}L^2_{\xi}}^2 \big)^{1/2}\\
&\leq CT^{1/8-\e}\big(\sum_N \int_{\xi}\int_{\tau}\big\langle \tau+\xi|\xi| \big\rangle^{-1-2\e}|g_N(\tau,\xi)|^2 \big)^{1/2}\\
&\leq CT^{1/8-\e} \big(\sum_N \|g_N \|_{L^2_{\xi}L^{\infty}_{\tau}}^2 \big)^{1/2},
\end{split}
\end{equation*}
which gives
\bna
|J_B|+|J_C| &\leq &CT^{1/8-\e} \big(\sum_N \|g_N \|_{L^2_{\xi}L^{\infty}_{\tau}}^2 \big)^{1/2}\|w \|_{X^{0,1/2}_T}\big(\|u \|_{L^4}+\|u \|_{X^{-1,1}_T}\big).
\ena
The result follows by duality.
\enp
Note that the two previous lemmas give Lemma \ref{bilinestim}.
\begin{lemme}
\label{lemmeexphigh}
Let $N\in \N^*$. We have the following estimates, uniformly in $N$.
\begin{equation*}
\begin{split}
\|P_{\geq N}(e^{i\frac{F}{2}}) \|_{L^{\infty}}&\leq \frac{C}{\sqrt{N}}\|F \|_{H^1},\\
\|P_{\geq N}(e^{i\frac{F_1}{2}}-e^{i\frac{F_2}{2}}) \|_{L^{\infty}}&\leq \frac{C}{\sqrt{N}}\big[1+\| F_1 \|_{H^1}+\| F_2 \|_{H^1}\big]\|F_1-F_2 \|_{H^1},\\
\end{split}
\end{equation*}
and 
\begin{equation*}
\|e^{i\frac{Q_N F_1}{2}}-e^{i\frac{Q_N F_2}{2}} \|_{L^{\infty}}\leq \frac{C}{\sqrt{N}}\big[1+\| F_1 \|_{H^1}+\| F_2 \|_{H^1}\big]\|F_1-F_2 \|_{H^1}.
\end{equation*}
\end{lemme}
\begin{remark}
Note that these estimates are very close to the estimate
\bna
\|P_{\geq 1}(e^{i\frac{F}{2}}) \|_{L^{\infty}([0,\lambda])}\leq C\|F_x \|_{L^2([0,\lambda])}
\ena
but uniform on the length $\lambda$, as used by Molinet \cite{molinetBOT}.

This gain comes from the fact that the inequality is subcritical. Roughly speaking, $L^{\infty}$ scales like $H^{1/2}$ which gives a gain at high frequency.
\end{remark}
\bnp
We will mainly use the following estimate, which is only the Sobolev embedding at high frequency.
\begin{equation}\label{sobhigh}
\begin{split}
\|P_{\geq N}(f)\|_{L^{\infty}}&\leq \sum_{n=N}^{+\infty} \big|\hat{f}(n)\big|\leq
\big(\sum_{n=N}^{+\infty} \frac{1}{n^2}\big)^{1/2}\big(\sum_{n=N}^{+\infty} n^2\big|\hat{f}(n)\big|^2\big)^{1/2}\leq \frac{C}{\sqrt{N}}\|f\|_{H^1}.
\end{split}
\end{equation}
We apply this estimate to $f=e^{i\frac{F_1}{2}}-e^{i\frac{F_2}{2}}$.
\begin{equation*}
\begin{split}
&\|P_{\geq N}(e^{i\frac{F_1}{2}}-e^{i\frac{F_2}{2}}) \|_{L^{\infty}}\\
&\leq \frac{C}{\sqrt{N}}\|e^{i\frac{F_1}{2}}-e^{i\frac{F_2}{2}} \|_{H^1}\leq \frac{C}{\sqrt{N}}\big(\|e^{i\frac{F_1}{2}}-e^{i\frac{F_2}{2}} \|_{L^2}+\|(\partial_x F_1)e^{i\frac{F_1}{2}}-(\partial_x F_2)e^{i\frac{F_2}{2}} \|_{L^2}\big)\\
&\leq\frac{C}{\sqrt{N}}\big(\| F_1-F_2 \|_{L^2}+\| F_1 \|_{H^1}\min(1,\|F_1-F_2 \|_{H^1})+\|F_2\|_{H^1}\|F_1-F_2 \|_{H^1}\big)
\end{split}
\end{equation*}
where, for the last estimate, we have used estimate (\ref{estimlipsexp}) together with
\bna
(\partial_x F_1)e^{i\frac{F_1}{2}}-(\partial_x F_2)e^{i\frac{F_2}{2}}=(\partial_x F_1)\big(e^{i\frac{F_1}{2}}-e^{i\frac{F_2}{2}}\big)+e^{i\frac{F_2}{2}}\partial_x (F_1-F_2).
\ena
This proves the second inequality of the lemma. The first estimate of the lemma is a consequence of the previous inequality with $F_2=0$ using the 
fact that $P_{\geq N}(e^{i\frac{F}{2}})=P_{\geq N}(e^{i\frac{F}{2}}-1)$.

The last inequality is proved using the Gagliardo-Nirenberg inequality on $\T$, i.e.
\bna
\|e^{i\frac{Q_N F_1}{2}}-e^{i\frac{Q_N F_2}{2}} \|_{L^{\infty}}\leq \|e^{i\frac{Q_N F_1}{2}}-e^{i\frac{Q_N F_2}{2}} \|_{L^2}
+\|e^{i\frac{Q_N F_1}{2}}-e^{i\frac{Q_N F_2}{2}} \|_{H^1}^{1/2}\|e^{i\frac{Q_N F_1}{2}}-e^{i\frac{Q_N F_2}{2}} \|_{L^2}^{1/2}
\ena
The previous computation shows that
\bna
\|e^{i\frac{Q_N F_1}{2}}-e^{i\frac{Q_N F_2}{2}} \|_{H^1}&\leq &C \big[1+\| Q_NF_1 \|_{H^1}+\| Q_NF_2 \|_{H^1}\big]\|Q_NF_1-Q_NF_2 \|_{H^1}\\
&\leq &C\big[1+\| F_1 \|_{H^1}+\| F_2 \|_{H^1}\big]\|F_1-F_2 \|_{H^1}
\ena
But, the mean value theorem gives
\bna
\|e^{i\frac{Q_N F_1}{2}}-e^{i\frac{Q_N F_2}{2}} \|_{L^2}\leq \|Q_N(F_1-F_2) \|_{L^2}\leq \frac{C}{N}\|F_1-F_2 \|_{H^1}.
\ena
\enp


\end{document}